\documentclass[12pt,a4paper,twoside]{amsart}
\usepackage[left=30mm,right=30mm,top=30mm,bottom=35mm]{geometry}

\usepackage{enumitem}
\usepackage{labelfig}

\usepackage{epsfig}
\usepackage{epstopdf}

\usepackage{amsmath}
\usepackage{amssymb}
\usepackage{latexsym}

\newcommand{\co}{\colon\thinspace} 
\newcommand{\fnote}[1]{\footnote{\small sharp1}}

\newcommand{\N}{{\mathbb N}}
\newcommand{\Z}{{\mathbb Z}}
\newcommand{\R}{{\mathbb R}}
\newcommand{\Q}{{\mathbb Q}}

\newcommand{\T}{{\mathbb T}}
\newcommand{\arcsinh}{{\mathrm{arcsinh}}}

\newcommand{\spt}{\mathrm{supp}}

\newcommand{\inter}{\mbox{Int}}

\newtheorem{theorem}{Theorem}[section]
\newtheorem{proposition}[theorem]{Proposition}

\newtheorem{definition}[theorem]{Definition}
\newtheorem{lemma}[theorem]{Lemma}
 
\newtheorem{conjecture}[theorem]{Conjecture} 
\newtheorem{question}[theorem]{Question} 

\title{On the homology length spectrum of surfaces}
\author{Daniel Massart, Hugo Parlier}
\date{\today}
\begin{document}

\begin{abstract}
On a surface with a Finsler metric, we investigate the asymptotic growth of the number of closed geodesics of length less than $L$ which minimize length among all geodesic multicurves in the same homology class. An important class of surfaces which are of interest to us are hyperbolic surfaces. 
\end{abstract}
\maketitle

\section{Introduction}
\subsection{The questions we ask}

Let $(M,m)$ be a closed, orientable manifold of dimension two, equipped with a Finsler metric. We are interested in the asymptotic growth, as $T$ grows, of a certain set of closed geodesics of length less than $T$.

Let us denote $\mathcal{G}_0$ the set of all closed geodesics of $(M,m)$ and for $T \in \R^*_+$ by $\mathcal{G}_0(T)$ the subset of $\mathcal{G}_0$ which consists of geodesics of length less than or equal to $T$. When $m$ is a Riemannian metric of pinched negative curvature, Margulis \cite{Margulis} showed that the cardinality of $\mathcal{G}(T) $ grows like $e^T /T$. 

Similarly, for a fixed homology class $h$  in $H_1(M,\Z)$, one can count the number of closed geodesics with homology $h$ and of length less than $L$. The asymptotic growth 
of this number has also been studied, for instance by Philips and Sarnak \cite{Philips-Sarnak}. 

Here we are interested in a problem which is in some sense orthogonal to the latter: instead of picking a homology class and counting all closed geodesics therein, we consider all homology classes, and associate to each one (at most) one closed geodesic. We then estimate of the asymptotic growth of this set of closed geodesics. 

To be precise we need to introduce the stable norm of $m$. This is the function which maps each homology class $h$ in $H_1(M,\Z)$ to the smallest possible length of a union of geodesics representing $h$. This union will always be a weighted multicurve (disjoint union of simple closed geodesics, possibly with multiplicity).
 Recall that since $M$ is orientable, $H_1(M,\Z)$ embeds as a lattice in $H_1(M,\R)$. The map thus defined extends to a norm on $H_1(M,\R)$, called stable norm and denoted $\| \|_g$ (\cite{Federer, GLP, gafa, nonor}). We denote by $\mathcal{B}_1$ its unit ball. We say a geodesic multicurve is minimizing if it minimizes the length among all multicurves within the same homology class. 

There are multiple counting problems one could investigate here. The first one that comes to mind is to find asymptotic estimates for the number of elements of $H_1(M,\Z)$ with stable norm less than $T$ when $T$ goes to infinity. By Minkovski's theorem, this is easily seen to be $\approx \mbox{Vol}(\mathcal{B}_1) T^{2g}$
where $g$ is the genus of $M$ (and thus the dimension of $H_1(M,\R)$ is $2g$). The quantity $\mbox{Vol}(\mathcal{B}_1)$ is the volume of the unit ball with respect to the following volume form. As $M$ is orientable, the algebraic intersection of oriented closed curves induces a symplectic form on $H_1(M,\R)$, which we denote $\inter (.,.)$. The $g$-th power of this symplectic form is a volume form on $H_1(M,\R)$ with respect to which the integer lattice $H_1(M,\Z)$ has determinant one. The volume $\mbox{Vol}(\mathcal{B}_1)$ is meant with respect to this volume form. 

However, we argue that this is not necessarily the most relevant counting problem when it comes to surfaces of genus greater than one. Indeed, it is proved in \cite{gafa} that infinitely many homology classes in $H_1(M,\Z)$ have a minimal representative which is a multicurve with exactly $g$ connected components. Furthermore, if a simple closed geodesic occurs as a connected component of a geodesic multicurve, then in fact it occurs as a connected component of infinitely many geodesic multicurves \cite{gafa}. Therefore, when counting geodesic multicurves, there is a lot of redundancy and it is thus of interest to count these connected components. 

It can be seen from \cite{gafa, nonor} that a closed geodesic $\gamma$ is a connected component of some minimizing multicurve if and only if it minimizes the length in its homology class. That is, $\|\left[\gamma\right]\| = \ell_m (\gamma)$, or equivalently, the homology class $\left[\gamma\right]/ \ell_m (\gamma)$ lies on the unit sphere $\partial \mathcal{B}_1$ of the stable norm. More precisely, $\gamma$ is a connected component of some minimizing multicurve if and only if the homology class $\left[\gamma\right]/ \ell_m (\gamma)$ is a vertex of the stable norm, that is, the unit ball has an open set of supporting hyperplanes at $\left[\gamma\right]/ \ell_m (\gamma)$.

So the objects we shall focus on are the closed geodesics which minimize length in their homology class (thus among all multicurves with the same homology). We denote the set of these by $\mathcal{G}$. For the sake of brevity we shall call such closed geodesics minimizing in their homology classes, or homologically minimal. 
Our purpose here is to find asymptotic estimates, when $T \rightarrow + \infty$, for the cardinality $N(T)$ of the set
$$
H(T):= \{ \left[\gamma\right] \co \gamma \in \mathcal{G},\ \ell_m (\gamma) \leq T \}.
$$
Again by Minkovski's theorem, $N(T)$ is bounded from above by $ \mbox{Vol}(\mathcal{B}_1) T^{2g}$. How far below this estimate the actual growth is, is a measure of how few closed geodesics minimize the length in their homology class, or how many homology classes are minimized by (non-connected) multicurves. 

Observe that when $M$ is not orientable, the unit ball of the stable norm could be a finite polyhedron (see \cite{nonor}), in which case there are only finitely many homology classes of elements of $\mathcal{G}$. This however never happens for orientable surfaces, so the asymptotic counting problem will always be non-trivial in this case. 

Although we have scant evidence as to what the answer might be, we ask the following question to get the ball rolling.
\begin{question}\label{question 1}
Does $N(T)$ grow quadratically for all $(M,m)$? 
\end{question}

When $M$ is a torus, this follows from Minkovski's theorem. Although our setup does not allow for this, the same questions can be asked when the surface is hyperbolic torus with a cusp, a case well investigated by McShane and Rivin \cite{McS-R1, McS-R2} where among things they prove asymptotic quadratic growth and a bound on the error term. 
Furthermore, minimizing geodesics do not pass through separating ``long thin necks" (see Lemma \ref{lem:longthinnecks} for the exact definition and statement). When $M$ is such that these ``necks" cut the surfaces into parts of genus at most $1$, $N(T)$ becomes a finite sum of functions which grow quadratically and thus also grows quadratically. This is discussed in more detail in section \ref{sec:longthicknecks}.

There is a connection between the problem we investigate and the question of counting closed trajectories of a polygonal billiard table with angles commensurable to $\pi$. Such closed trajectories correspond to closed geodesics of a flat surface with cone point singularities (see \cite{KMS}). The closed geodesics which come from closed billiard trajectories are easily seen to be minimizing in their homology classes (but not all minimizing geodesics come this way). A positive answer to Question \ref{question 1} would thus provide a generalization of Masur's quadratic upper bound on the number of strips of parallel closed billiard trajectories \cite{Masur}. Furthermore, Masur's bound provides further evidence that the correct bound is indeed quadratic. 

\subsection{The answers we give}
Now that we have introduced the problem, let us lay out what we do in this paper. We want to organize the set $\mathcal{G}$ into families which grow quadratically and, for each family, provide a geometric interpretation of the quadratic constant. For simplicity, and for the remainder of the introduction, we suppose that $M$ is of genus two.

Let us fix an element of $\mathcal{G}$, that is, a simple, closed, homologically minimizing geodesic $\gamma$. Denote by $\mathcal{G}_{\gamma}$ the set of all elements $\delta$ of $\mathcal{G}$, such that the reunion of $\gamma$ and $\delta$ is a homologically minimizing multicurve. The sets $\mathcal{G}_{\gamma}$, when $\gamma$ ranges over $\mathcal{G}$, are the families into which we organize (but {\it not} partition) $\mathcal{G}$. So, denoting
\begin{itemize}[itemsep=2ex,leftmargin=0.5cm]
\item $\mathcal{G}_{\gamma}(L)$ the set of elements of $\mathcal{G}_{\gamma}$ of length $\leq L$
\item $H_{\gamma}(L)$ the set of homology classes of elements of $\mathcal{G}_{\gamma}(L)$
\item $N_{\gamma}(L)$ the cardinality of $H_{\gamma}(L)$,
\end{itemize}
we want to prove that $N_{\gamma}(L)/L^2$ has a limit when $L$ goes to infinity, and provide a geometric interpretation for this limit. 

Since closed geodesics which minimize the length in their homology class correspond to vertices of the unit ball of the stable norm, the approach we propose to our counting problem is to understand the set of vertices of the unit ball. 

First let us observe that for any $\delta$ in $\mathcal{G}_{\gamma}$, since $\gamma$ and $\delta$ are disjoint, we have $\inter (\left[\gamma\right],\left[\delta\right])=0$, that is, $\left[\delta\right]$ lies in the symplectic orthogonal $\left[\gamma\right]^{\perp}$ of $\left[\gamma\right]$. As the genus of $M$ is $2$, this symplectic orthogonal is a $3$-dimensional subspace of $H_1(M,\R)$, which contains $\left[\gamma\right]$. Furthermore, the kernel of the restriction to $\left[\gamma\right]^{\perp}$ of the symplectic form $\inter(.,.)$ is the straight line generated by $\left[\gamma\right]$. So the quotient space $\left[\gamma\right]^{\perp}/ \left[\gamma\right]$ inherits a symplectic structure, which will be useful in the sequel. 

Moreover, by \cite{gafa,nonor} (see Theorem \ref{thm-gafa}) we know the following facts :
\begin{itemize}[itemsep=2ex,leftmargin=0.5cm]
\item
for any $\delta$ in $\mathcal{G}_{\gamma}$, the unit ball $\mathcal{B}_1$ of the stable norm has an edge joining $\left[\gamma\right]/ \| \left[\gamma\right]\|$ to $\left[\delta\right]/ \| \left[\delta\right]\|$
\item
any edge of $\mathcal{B}_1$ starting from $\left[\gamma\right]/ \| \left[\gamma\right]\|$ is contained in $\left[\gamma\right]^{\perp}$
\item
for any non-zero element $v$ of the quotient space $\left[\gamma\right]^{\perp}/ \left[\gamma\right]$, there exists an edge $e$ of $\mathcal{B}_1$ starting from $\left[\gamma\right]/ \| \left[\gamma\right]\|$, which projects to the straight line segment $\left[ 0, \lambda v \right]$ in $\left[\gamma\right]^{\perp}/ \left[\gamma\right]$, for some positive $\lambda$.
\end{itemize}
It might be useful for the reader to get a mental picture of the set of edges of $\mathcal{B}_1$ starting from $\left[\gamma\right]/ \| \left[\gamma\right]\|$. The intersection of the unit sphere $\partial \mathcal{B}_1$ with $\left[\gamma\right]^{\perp}$, which has dimension $3$, is topologically a two-sphere. The homology class $\left[\gamma\right]/ \| \left[\gamma\right]\|$ is a vertex of this topological two-sphere. There is a neighborhood $V$ of $\left[\gamma\right]/ \| \left[\gamma\right]\|$ in $\partial \mathcal{B}_1 \cap \left[\gamma\right]^{\perp}$ such that for any homology class $h$ in $V$, the straight segment joining $h$ to $\left[\gamma\right]/ \| \left[\gamma\right]\|$ is contained in $\partial \mathcal{B}_1 \cap \left[\gamma\right]^{\perp}$. We denote by $\tilde{\mathcal{E}}(\gamma)$ the reunion of all edges of $\mathcal{B}_1$ starting from $\left[\gamma\right]/ \| \left[\gamma\right]\|$.
 Thus $\tilde{\mathcal{E}}(\gamma)$ projects to a compact neighborhood $\mathcal{E}(\gamma)$ of $0$ in the quotient space $\left[\gamma\right]^{\perp}/ \left[\gamma\right]$.

For instance, when $M$ is a surface with long thin necks (as described in Section \ref{sec:longthicknecks} and illustrated in Figure \ref{fig:genus5}), for any $x,y \in \R$, the homology class
$$
\frac{x \left[ \delta\right] + y \left[ \beta\right]}{\| x \left[ \delta\right] + y \left[ \beta\right] \|}
$$

is an endpoint of some edge starting from $\left[\gamma\right]/ \| \left[\gamma\right]\|$. We can identify the quotient space $\left[\gamma\right]^{\perp}/ \left[\gamma\right]$ with the subspace of $H_1(M,\R) $ generated by $\left[\delta\right]$ and $\left[\beta\right]$. Then the neighborhood $\mathcal{E}(\gamma)$ is simply the set
$$
\{ \frac{x \left[ \delta\right] + y \left[ \beta\right]}{\| x \left[ \delta\right] + y \left[ \beta\right] \|} \co x,y \in \R\}
$$
which is also the unit ball of the stable norm of the torus with one hole obtained by cutting $M$ along the short separating geodesic in the middle of the neck and discarding the left-hand side. For a general surface $M$ there is no reason why all endpoints of edges starting from $\left[\gamma\right]/ \| \left[\gamma\right]\|$ should be co-planar, which is why we introduce the quotient space $\left[\gamma\right]^{\perp}/ \left[\gamma\right]$.

Proposition \ref{injection F Gamma} says that under some genericity assumption on the metric $m$, the elements of $\mathcal{G}_{\gamma}$ are in 1-to-1 correspondance with the integer vectors in the vector space $\left[\gamma\right]^{\perp}/ \left[\gamma\right]$. Furthermore, for any $L \geq 0$, the elements of $\mathcal{G}_{\gamma}(L)$ are in 1-to-1 correspondance with the integer vectors in 
$$
L \mathcal{E}(\gamma) = \{ tx \co t \in \left[ 0,L \right], x \in \mathcal{E}(\gamma) \}.
$$
A variation on the classical Minkovski theorem (Proposition \ref{appendice}) then says that $N_{\gamma}(L)/L^2$ converges to the volume, with respect to the symplectic structure on the quotient space $\left[\gamma\right]^{\perp}/ \left[\gamma\right]$, of the compact set $\mathcal{E}(\gamma)$.

For a surface of genus $g > 2$, we have to adjust our strategy a little bit : instead of fixing an element $\gamma$ of $\mathcal{G}$, we fix $g-1$ elements $\gamma_1,\ldots \gamma_{g-1}$ 
 elements of $\mathcal{G}$, whose reunion $\Gamma$ is a minimizing multicurve, and we consider the subset $\mathcal{G}_{\Gamma}$ of $\mathcal{G}$ which consists of closed geodesics $\gamma$ such that the reunion of $\gamma$ and $\Gamma$ is a minimizing multicurve. Then we prove that $\mathcal{G}_{\Gamma}$, as $\mathcal{G}_{\gamma}$ in the genus 2 case, grows quadratically, again under some genericity assumption on the metric $m$.
The genericity hypothesis we alluded to in the previous paragraph is that in every homology class there is at most one minimizing multicurve. It would be interesting to know whether this hypothesis is truly necessary, and if it is, how restrictive it is. The word ``genericity" seems to indicate that metrics not satisfying this hypothesis are Baire meagre in the set of all metrics. However, we do not have a proof of this fact. It is true, however, in the smaller set of metrics of constant negative curvature, as outlined in Section \ref{sec:genericity}. In the larger setting of Finsler metrics, we believe the genericity could be proved using the machinery of \cite{BC}.

Now we would like to know how much our result really says about Question \ref{question 1}. Again, for simplicity we consider the genus two case. If, instead of fixing an element of $\mathcal{G}$, we consider a family $(\gamma_i)_{i \in I}$ of elements of $\mathcal{G}$, can we say something about the growth of the reunion
 $\bigcup_{i \in I} \mathcal{G}_{\gamma_i} \ $ ? Obviously we can if the family $(\gamma_i)_{i \in I}$ is finite, but even in that case the interpretation of the quadratic constant is not clear since for each $\gamma_i$ we have to use the symplectic structure of the quotient space $\left[\gamma_i\right]^{\perp}/ \left[\gamma_i\right]$.

It would be interesting to interpret the quadratic constants directly in terms of the symplectic structure of $H_1(M,\R)$. For this we suggest the following procedure.\

{\it Step 1}: Prove that for each $\gamma \in \mathcal{G}$, the set $\mathcal{F}_{\gamma} $ defined as the closure in $H_1(M,\R)$ of the set 
$$
\{ t\left[ \delta \right] \co \delta \in \mathcal{G}_{\gamma},\ t \in \left[0,1 \right] \}
$$
is rectifiable, so we can evaluate its symplectic area with respect to the symplectic form $\inter(.,.)$. Denote it $\Omega (\gamma)$.

{\it Step 2}: Prove that $\Omega (\gamma)$ equals the area, with respect to the symplectic structure on the quotient space $\left[\gamma\right]^{\perp}/ \left[\gamma\right]$, of the compact set $\mathcal{E}(\gamma)$.

What is interesting about this point of view is that assuming both points above are true, if we have a countable family $(\gamma_i)_{i \in I}$ of elements of $\mathcal{G}$, the reunion $\bigcup_{i \in I} \mathcal{G}_{\gamma_i} \ $ is again rectifiable, so we can evaluate its symplectic area.

Recall that the set of homology class of elements of $\mathcal{G}$ is countable. We propose the following.
\begin{conjecture}\label{conj}
Let 
\begin{itemize}[itemsep=2ex,leftmargin=0.5cm]
 \item $M$ be a closed, orientable, surface of genus $g$, equipped with a generic Finsler metric $m$
 \item $\mathcal{F}$ be the closure in $H_1(M,\R)$ of the set $ \{ t\left[ \delta \right] \co \delta \in \mathcal{G},\ t \in \left[0,1 \right] \}$
 \item $ \mathcal{G}_T $ be the set of elements of $\mathcal{G}$ with length $\leq T$
\item $H(T)$ the set of homology classes of elements of $\mathcal{G}(T)$
\item $N(T)$ the cardinality of $H(T)$,
 \end{itemize}
Then 
\begin{itemize}[itemsep=2ex,leftmargin=0.5cm]
\item $\mathcal{F}$ is rectifiable. We denote $ \Omega \left(\mathcal{F}\right)$ its symplectic area with respect to the intersection form $\inter(.,.)$.
\item we have the following asymptotic estimate :
$$
\lim_{T \rightarrow + \infty} \frac{N(T)}{T^2} = \Omega \left(\mathcal{F}\right).
$$
\end{itemize}
\end{conjecture}
Of course, even if the conjecture is true, it doesn't say much if $\Omega \left(\mathcal{F}\right)$ is infinite. Therefore our aim is now to find sufficient conditions for $\Omega \left(\mathcal{F}\right)$ to be finite. One such condition is when $M$ has $g-1$ thin long necks. The question is, how far from this rather obvious hypothesis can we go?\\

\textbf{Acknowledgements.}

The first author was partially supported by the ANR grant``Hamilton-Jacobi et th\'eorie KAM faible" and the second author was supported by Swiss National Science Foundation grant number PP00P2$\_$128557. The authors thank Jean Saint Pierre for the statement of Proposition \ref{appendice}.
\section{Notations and preliminaries}
\subsection{Minimizing measures}
We denote by $\mathcal{M}$ the set of Borel measures on the unit tangent bundle $T^1 M$ of $(M,m)$, not necessarily normalized (i.e. the total mass need not be 1), which are invariant under the geodesic flow of $m$. 
If $\mu \in \mathcal{M}$, we denote 
$$
m(\mu) := \int_{T^1 M} 1 d \mu
$$
so $m(\mu)$ is both the total mass of $\mu$, and, when $\mu$ is supported by a closed geodesic $\alpha$, the length of $\alpha$, thus the notation $m(\mu)$ is quite convenient. 

By \cite{Mather91} (see also \cite{nonor}), for any $\mu$ in $\mathcal{M}$ and any $C^1$ function $f$ on $M$ we have
\begin{equation}\label{mesures fermees}
 \int_{T^1 M} df d \mu = 0.
 \end{equation}
Thus the measures in $\mathcal{M}$ have a well-defined homology class in $H_1(M,\R)$. For instance if $\mu$ is supported by a closed orbit of the geodesic flow $(\gamma, \dot\gamma)$, $\gamma$ a closed geodesic, then the homology class of $\mu$ is just that of $\gamma$. 

 We say a measure $\mu$ in $\mathcal{M}$ is minimizing when it is minimizing in its homology class $\left[\mu\right]$, i.e., when
$$
m(\mu) = \|\left[\mu\right]\|. 
$$
One of the reasons for considering invariant measures (rather than just geodesic multicurves) is that there may be homology classes in $H_1(M,\R)$ for which there is no minimizing multicurve. In fact, when $M$ is an orientable surface, such classes always exist but there still always exists a minimizing measure \cite{Mather91, nonor}. 

\begin{lemma}\label{lemme combinaisons lineaires}
Let $\mu, \nu \in \mathcal{M}$ be such that $\mu + \nu$ is minimizing in its homology class. Then for any $a,b \geq 0$, $a\mu + b\nu$ is minimizing in its homology class.
\end{lemma}
\proof 
Take a supporting hyperplane at $\left[\mu+\nu\right]/\|\left[\mu+\nu\right]\|$ to the ball $\mathcal{B}_1$, that is, a cohomology class $c \in H^1(M,\R)$ such that 
\begin{eqnarray}
\langle c,h \rangle & \leq & \|h\| \ \forall h \in H_1 (M,\R)\label{c support} \\ 
\langle c, \left[\mu+\nu\right]\rangle & = & \|\left[\mu+\nu\right]\|.
\end{eqnarray}
For any closed $1$-form $\omega$ on $M$ such that $\left[\omega\right]=c$, we have, by Equation (\ref{mesures fermees}),
\begin{equation}
\langle c, \left[\mu+\nu\right]\rangle = \int_{T^1 M} \omega d (\mu+\nu)= \int_{T^1 M} \omega d \mu \ +\ \int_{T^1 M} \omega d \nu .
\end{equation}
Since $\mu + \nu$ is minimizing in its homology class, we have 
\begin{equation}
 \|\left[\mu+\nu\right]\| = m(\mu+\nu)= m(\mu)+m(\nu).
 \end{equation}
so 
\begin{equation}\label{lemme 2.1, egalite}
 \int_{T^1 M} \omega d \mu \ +\ \int_{T^1 M} \omega d \nu= m(\mu)+m(\nu).
 \end{equation}
On the other hand, by the inequality (\ref{c support}), we have 
\begin{eqnarray*}
\langle c, \left[\mu\right]\rangle & \leq & \|\left[\mu\right]\| \leq m(\mu) \\
\langle c, \left[\nu\right]\rangle & \leq & \|\left[\nu\right]\|\leq m(\nu).
\end{eqnarray*}
Those two inequalities sum up to the equality (\ref{lemme 2.1, egalite}), so both inequalities are equalities : 
\begin{eqnarray*}
 \int_{T^1 M} \omega d \mu &=& m(\mu) \\
 \int_{T^1 M} \omega d \mu &=& m(\nu).
\end{eqnarray*} 
Now take $a,b \geq 0$. We have 
\begin{eqnarray*}
\langle c, a \left[\mu\right] + b \left[\nu\right] \rangle &=& a \langle c, \left[\mu\right] \rangle + b \langle c, \left[\nu\right] \rangle \\
&=& a m(\mu) +b m(\nu) \\
&=& a \|\left[\mu\right]\| +b \|\left[\nu\right]\| \\
& \geq & \| a \left[\mu\right] + b\left[\nu\right]\|
\end{eqnarray*}
where the last inequality is just the triangle inequality. 
On the other hand, by inequality (\ref{c support}), we have 
$$
\langle c, a \left[\mu\right] + b \left[\nu\right] \rangle \leq \| a \left[\mu\right] + b\left[\nu\right]\|
$$
hence 
$$
a m(\mu) +b m(\nu)= \langle c, a \left[\mu\right] + b \left[\nu\right] \rangle = \| a \left[\mu\right] + b\left[\nu\right]\|
$$
which says that $a\mu + b\nu$ is minimizing in its homology class.
\qed
\subsection{Faces of the stable norm}
We will need Proposition 5.6, Theorem 6.1, Proposition 5.4, and Lemma 5.5 of \cite{nonor} so we recall their statements below. 

We say a homology class $h$ in $H_1(M,\R)$ is
\begin{itemize}[itemsep=2ex,leftmargin=0.5cm]
 \item integer if $h \in H_1 (M,\Z) \subset H_1 (M,\R)$
 \item rational if there exists $h' \in H_1 (M,\Z)$ and $n \in \N^*$ such that $h'=n h$
 \item 1-irrational if there exists $h' \in H_1 (M,\Z)$ and $\lambda \in \R$ such that $h=\lambda h'$.\\
\end{itemize} 

\begin{proposition}\label{rational}
Let $M$ be a closed (possibly non-orientable) surface with a Finsler metric $m$.
If $h$ is a 1-irrational homology class and $\mu$ is an $h$-minimizing measure, then the support of $\mu$ consists of periodic orbits.\\
\end{proposition}

\begin{theorem}\label{thm-gafa}
Let \begin{itemize}[itemsep=2ex,leftmargin=0.5cm]
 \item $M$ be a closed orientable surface endowed with a Finsler metric $m$,
 \item $\gamma_1, \ldots \gamma_{k}$ be simple closed geodesics such that the formal sum $\Gamma := \gamma_1 + \ldots + \gamma_{k}$ is a minimizing multicurve,
 \item $h_0$ be $\left[\gamma_1\right]+ \ldots +\left[\gamma_k \right]$.
\end{itemize}
 For all $h \in V(\Gamma)^\perp$, there exists $s(h_0,h)>0$ such that the subset of the unit sphere $\partial \mathcal{B}_1$
$$
 \left\{\frac{ h_0+ sh}{|| h_0+sh||} \, \co s \in \left[0,s(h_0,h)\right] \right\}
$$
is a straight segment.\\
\end{theorem}

\begin{proposition}\label{two-sided}
Let $\gamma_2$ be a closed, simple, two-sided geodesic on a closed (possibly non-orientable) surface $M$ endowed with a Finsler metric $m$. There exists a neighborhood $V_2$ of $\left(\gamma_2,\dot\gamma_2\right)$ in $T^{1}M$ such that if $\gamma$ is a simple closed geodesics entering $V_2$ (resp. leaving $V_2$) then
\begin{itemize}[itemsep=2ex,leftmargin=0.5cm]
 \item
 either $\gamma$ is a closed geodesic homotopic to $\gamma_2$
 \item
or $\gamma$ is asymptotic to a closed geodesic homotopic to $\gamma_2$
 \item
 or $\gamma$ intersects $\gamma_2$, and all intersections have the same sign with respect to some orientation of $p (V_2)$, where $p$ denotes the canonical projection $TM \rightarrow M$.\\
 \end{itemize}
\end{proposition}

\begin{lemma}\label{asymptote_fermee}
Let $M$ be a closed (possibly non-orientable) surface with a Finsler metric.
If a geodesic $\gamma$ is asymptotic to a simple closed geodesic, then $(\gamma,\dot{\gamma})$ is not in the support of any minimizing measure.
\end{lemma}

\subsection{Genericity}\label{sec:genericity}
\begin{definition}\label{def:genericity}
We say $(M,m)$ is generic if, for every $h \in H_1(M,\Z)$, there exists a unique minimizing multicurve in the homology class $h$.
\end{definition}
We warn the reader that by using the word "generic" we are abusing terminology a little bit. Indeed by \cite{Mane}, the property of having only one minimizing measure for each integer homology class is residual in the set of Tonelli Lagrangians. However, this is still a much larger set than the set of all Finsler metrics. In the smaller set of hyperbolic surfaces,  this property is true for a Baire dense set of surfaces. (We are considering the usual topology on the moduli space of hyperbolic metrics on a surface of fixed genus.)

\begin{proposition}
The set of hyperbolic surfaces without unique minimizing multicurves for every homology class is Baire meagre. 
\end{proposition}

\begin{proof}
The ingredients for proving this can be found in \cite[Lemma 4.1]{McShane-Parlier} so we only outline the argument here as the details are identical. 

There are  countably many topological types of integer weighted multicurves on a finite type hyperbolic surface. Any given hyperbolic metric will have a unique geodesic representative for a given topological type of multicurve, so to each type one associates a function over {\it Teichm\"uller space}, the space of {\it marked} hyperbolic metrics of given type. Now given two such multicurves, we consider the subspace of Teichm\"uller space where the two geodesic representatives are of the same length. This subspace is the zero set for the difference of the length functions and length functions of this type are analytic with respect to the analytic structure of Teichm\"uller space. If the zero set of this difference function was contained in an open subset, it would be all of Teichm\"uller space but it is not difficult to prove that this is only possible if the topological types of the multicurves were identical to begin with. Furthermore, as these sets are zero sets for analytic functions, they are closed subspaces. One now concludes by applying Baire's category theorem to the countable collection of these subspaces. 
\end{proof}

Note that when a surface is generic, any minimizing multicurve has at most $g:= \mbox{genus}(M)$ connected components. All the faces of $\mathcal{B}_1$ are simplices. 

\subsection{Long thin necks and quadratic growth}\label{sec:longthicknecks}

We briefly describe a situation where we can guarantee quadratic growth. We begin with a lemma which says that minimizing geodesics do not cross separating  geodesics with a sufficiently wide collar around them.

\begin{lemma}\label{lem:longthinnecks}
Let $\gamma$ be a separating closed geodesic. Assume there exists $r >0$ such that the
set
$$
V_r (\gamma) = \{ x \in M \co d(x,\gamma) \leq r \}
$$
is homeomorphic to an annulus, and its boundary components are $C^1$ closed curves $\gamma_1$ and $\gamma_2$, such that for $i=1,2$ we have $2 d ( \gamma_i, \gamma) \geq  l(\gamma_i)$.

Then, for any  homologically minimal closed geodesic $\alpha$, we have $\alpha \cap \gamma = \emptyset$.
\end{lemma}
\proof
Let $\alpha$ be a homologically minimal closed geodesic, and assume $\alpha \cap \gamma \neq \emptyset$. Let $a$ be an arc of $\alpha$ contained in $V_r (\gamma)$, intersecting $\gamma$,  with endpoints $x_a$ and $y_a$ on the boundary of $V_r (\gamma)$.

{\it First case:} Assume $x_a$ and $y_a$ are on the same boundary component (say, $\gamma_1$) of $V_r (\gamma)$.  Then we have 
$$
l(a) \geq 2 d(\gamma_1,\gamma) \geq  l(\gamma_1)
$$
so, replacing $a$ with an arc of $\gamma_1$ going from $x_a$ to $y_a$, we obtain a closed curve homologous to $\alpha$, and shorter, which contradicts the minimality of $\alpha$. 

{\it Second case:} Assume $x_a$ and $y_a$ are on different  boundary components of $V_r (\gamma)$ (say, $x_a$ on $\gamma_1$ and $y_a$ on $\gamma_2$). Then the algebraic intersection of $a$ and $\gamma$ is not zero. Since $\gamma$ is separating and $\alpha$ is a closed curve, the algebraic intersection of $\alpha$ and $\gamma$ is zero, so there exists another arc $b$ of $\gamma$, intersecting $\gamma$, with endpoints $x_b$ and $y_b$ on $\gamma_2$ and $\gamma_1$, respectively. Then we have 
$$
l(a)+l(b) \geq   2 \left( d(\gamma_1,\gamma)+ d(\gamma_2,\gamma)\right) \geq  l(\gamma_1)+ l(\gamma_2)
$$
so, replacing $a$ and $b$ with  arcs of $\gamma_1$ and $\gamma_2$ going from $x_a$ to $y_b$ and  from $x_b$ to $y_a$, respectively, we obtain a multicurve homologous to $\alpha$, and shorter, which contradicts the minimality of $\alpha$.

\begin{figure}[h]
\leavevmode \SetLabels
\L(.105*.93) $x_a$\\
\L(.18*.93) $y_a$\\
\L(.143*.76) $a$\\
\L(.315*.93) $x_a$\\
\L(.39*.93) $y_a$\\
\L(.61*.93) $x_a$\\
\L(.685*.93) $y_b$\\
\L(.605*.79) $a$\\
\L(.665*.79) $b$\\
\L(.635*.07) $y_a$\\
\L(.685*.07) $x_b$\\
\L(.81*.93) $x_a$\\
\L(.885*.93) $y_b$\\
\L(.84*.07) $y_a$\\
\L(.90*.07) $x_b$\\
\endSetLabels
\begin{center}
\AffixLabels{\centerline{\epsfig{file =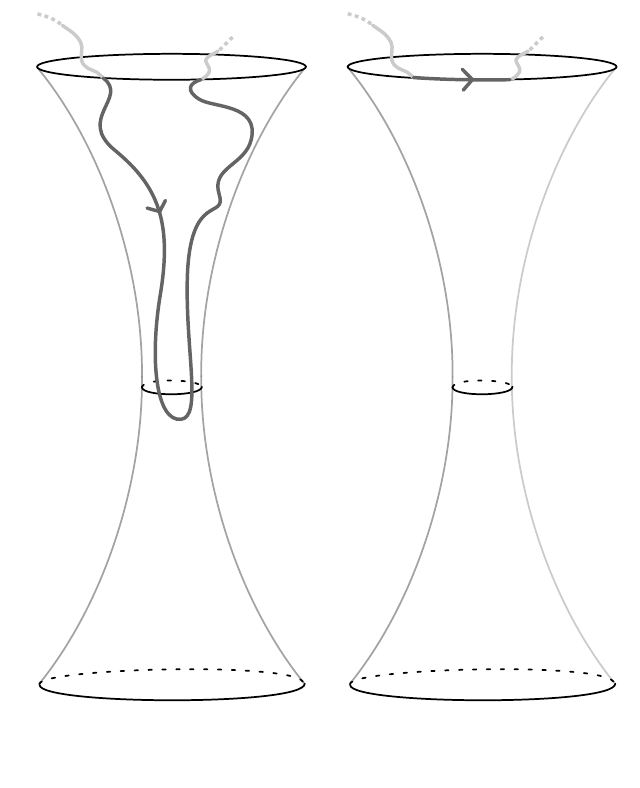,width=6.5cm,angle=0} \hspace{0.7cm} \epsfig{file =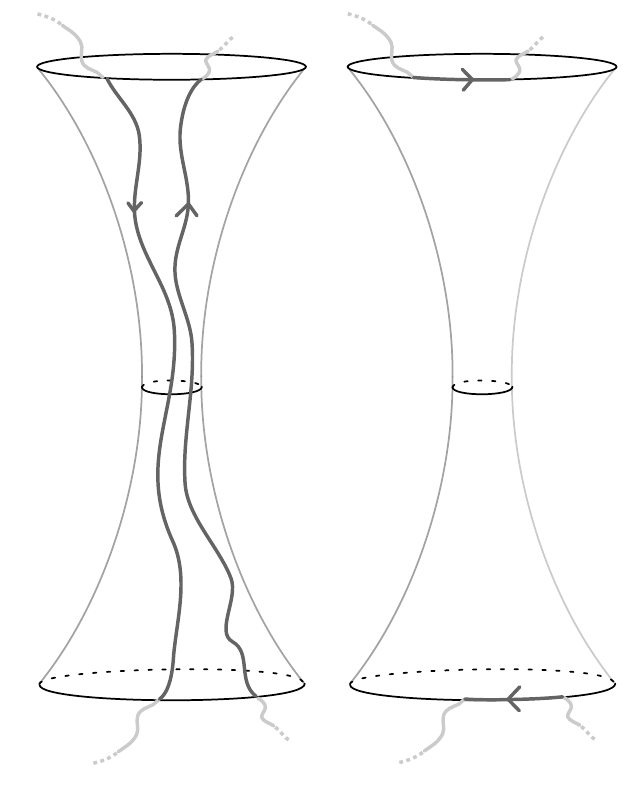,width=6.5cm,angle=0}}}
\vspace{-18pt}
\end{center}
\caption{First case on left and second case on right} \label{fig:cases}
\end{figure}
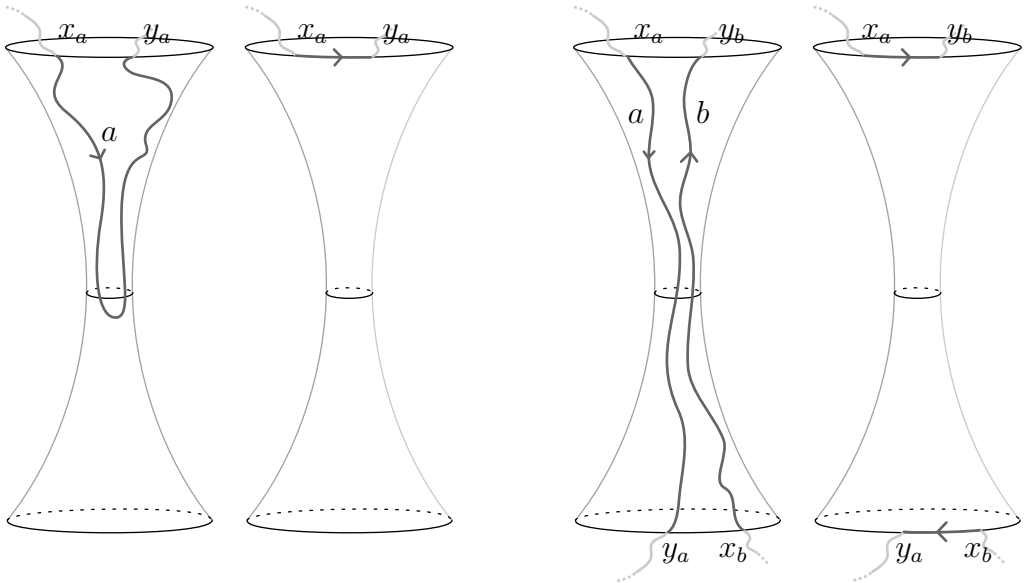

\qed
 
 \begin{definition} 
We say a separating closed geodesic  has a long thin neck if it satisfies the hypothesis of  lemma \ref{lem:longthinnecks} above.
\end{definition}

On a given surface we consider the set of curves with long thin necks as above. Observe that two such curves are necessarily disjoint thus this set of curves forms a multicurve. 

\begin{definition}
We say a Finsler surface $(M,m)$ is a giraffe if the complement in $M$ of the set of separating closed geodesics with long thin necks is a  collection of connected components where the genus of each component is at most $1$. 
\end{definition}
There are a certain number of possible topological configurations for a giraffe, depending on the number of boundary components of the connected components of  the complement in $M$ of the set of separating closed geodesics with long thin necks. In figure \ref{fig:genus5}, we give three  examples in genus $5$ of different configurations. 

\begin{figure}[h]
\leavevmode \SetLabels
\L(.23*.4) $\;$\\
\endSetLabels
\begin{center}
\AffixLabels{\centerline{\epsfig{file =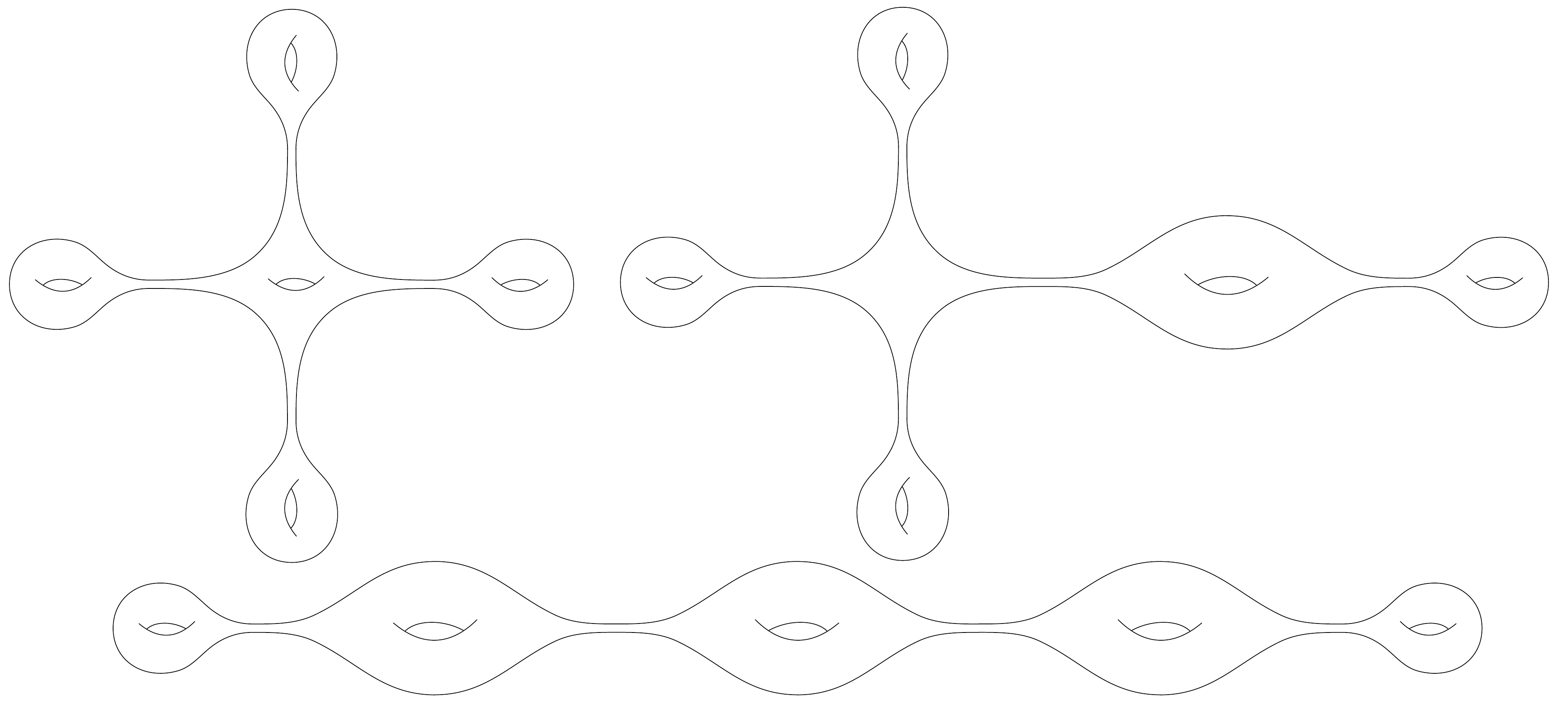,width=12.5cm,angle=0}}}
\vspace{-18pt}
\end{center}
\caption{Different configurations of genus 5 surfaces with long thin necks} \label{fig:genus5}
\end{figure}


Concrete examples of giraffes are given by hyperbolic surfaces with a collection of separating closed geodesics of length $\leq 2\, \arcsinh 1$ which cut the surface into a collection of tori and spheres (both with boundary). In this case the ``long thin necks" are just the disjoint collar regions around the curves, forced by the hyperbolic metric.

In the remainder of this section we assume $(M,m)$ be a giraffe of genus $g$. We denote by 
 $T_1,\ldots T_g$ the tori which are connected component of the complement in $M$ of the curves with long thin necks, and by  $H_i$ (for $i=1,\ldots g$) the image in $H_1(M,\R)$ of $H_1(T_i,\R)$ under the map induced by the canonical inclusion of $T_i$ into $M$. Beware that $H_i$ may not be isomorphic to $H_1(T_i,\R)$ if $T_i$ has more than one boundary component. 

\begin{lemma}
For each $i=1, \ldots g$, the subspace $H_i$ contains  a symplectic two-plane $S_i$ in $H_1(M,\R)$, and
$$
H_1(M,\R)= \bigoplus_{i=1}^g S_i
$$
\end{lemma}
\proof
Since $T_i$ is a torus with holes, we may find simple closed curves 
$\alpha_i$ and $\beta_i$ in $T_i$ which intersect exactly once. So we have, for all $i,j=1, \ldots g$,
$$
|\inter( (\left[ \alpha_i \right], \left[ \beta_j \right])|= \delta_{ij}
$$
and 
$$
|\inter( (\left[ \alpha_i \right], \left[ \alpha_j \right])|= 0
$$
 Therefore the homology classes $\left[ \alpha_i \right]$ and  $\left[ \beta_i \right]$,  $i=1, \ldots g$, form a symplectic basis of $H_1(M,\R)$. We now define $S_i$ as the linear span of $\left[ \alpha_i \right]$ and  $\left[ \beta_i \right]$ and the lemma is proved.
\qed

Now denote 
$$
\mathcal{G}_i = \{ \gamma \in \mathcal{G} \co \gamma \subset T_i\}.
$$
\begin{lemma}\label{lem G i et S i}
The elements of $\mathcal{G}_i$ are in one-to-one correspondence with the elements of 
$S_i \cap H_1(M,\Z)$.

\end{lemma}
\proof
First we pick an element $\gamma$ of $\mathcal{G}_i$ and prove its homology class lies in $S_i \cap H_1(M,\Z)$.

Since $\gamma$ is a closed curve its homology class lies in $H_1(M,\Z)$. Since $\left[ \alpha_i \right]$ and  $\left[ \beta_j \right]$,  $i,j=1, \ldots g$, form a  basis of $H_1(M,\R)$, we may write 
$$
\left[ \gamma \right]=\sum_{j=1}^g a_j \left[ \alpha_j \right]+b_j\left[ \beta_j \right].
$$
Now since $\gamma \subset T_i$, for any $j\neq i$, 
$$
\inter( \left[ \gamma \right],  \left[ \alpha_j \right]) = \inter( \left[ \gamma \right],  \left[ \beta_j \right]) =0
$$
but 
\begin{eqnarray*}
|\inter( \left[ \gamma \right],  \left[ \alpha_j \right])| &=& |b_j| \mbox{ and}\\
|\inter( \left[ \gamma \right],  \left[ \beta_j \right])| &=& |a_j|
\end{eqnarray*}
so $a_j=b_j=0$, that is, $\left[ \gamma \right] \in S_i$.

Conversely, let us pick an element $h$ of $S_i \cap H_1(M,\Z)$, and prove it is the homology class of an element of $\mathcal{G}_i$.

Let us write
$$
h=\sum_{j=1}^g a_j \left[ \alpha_j \right]+b_j\left[ \beta_j \right].
$$
Since $h\in S_i$, we have $a_j=b_j=0$ for any $j \neq i$.
 
By the genericity of $(M,m)$, there exists a unique $h$-minimizing multicurve $\gamma$.  We want to prove that $\gamma$ has only one connected component, and it is contained in $T_i$, so $\gamma$ is actually an element of $\mathcal{G}_i$.

Since $(M,m)$ is a giraffe, by Lemma \ref{lem:longthinnecks}, any connected component of the multicurve $\gamma$ is contained in some $T_j$. Let us assume that $\gamma$ has a connected component $\gamma_j$ contained in $T_j$, for $j\neq i$. Then the homology class of $\gamma_j$ is not zero, for if it were zero, then $\gamma \setminus \gamma_j$ would be another multicurve with homology $h$, shorter than $\gamma$, thus contradicting the minimality of $\gamma$. But if the homology class of $\gamma_j$ is  zero, then for some $j\neq i$, $a_j$ or $b_j$ is not zero, which we have already seen is impossible.  Therefore all connected components of $\gamma$ are contained in $T_i$. Now let us prove that $\gamma$ is connected. 

Assume $\gamma$ has two connected components $\gamma_1$ and $\gamma_2$, both contained in $T_i$. Then  $\gamma_1 \cap \gamma_2= \emptyset$, because $\gamma$ is a multicurve (disjoint union of simple closed curves). On the other hand, the homology classes of $\gamma_1$ and $\gamma_2$ may be written
\begin{eqnarray*}
\left[ \gamma_1 \right] &=& a_1 \left[ \alpha_i \right]+b_1\left[ \beta_i \right]\\
\left[ \gamma_2 \right] &=& a_2 \left[ \alpha_i \right]+b_2\left[ \beta_i \right]
\end{eqnarray*}
and since $\gamma_1 \cap \gamma_2= \emptyset$, the absolute value of the algebraic intersection of $\gamma_1$ and $\gamma_2$, which is $|a_1 b_2 - a_2 b_1|$, is zero. Therefore the homology classes $\left[ \gamma_1 \right]$ and $ \left[ \gamma_2 \right]$ are proportional, say  $\left[ \gamma_1 \right]= \lambda  \left[ \gamma_2 \right]$ for some $\lambda \in \Z$. Then if $l_m(\gamma_1)= |\lambda|l_m(\gamma_1)$, replacing $\gamma_1$ by $\lambda \gamma_2$ in $\gamma$, we find another $h$-minimizing multicurve, which contradicts the genericity of $(M,m)$; and $l_m(\gamma_1)\neq |\lambda|l_m(\gamma_1)$ contradicts the minimality of $\gamma$. Thus $\gamma$ has only one connected component, that is, $\gamma \in \mathcal{G}_i$. 

So the homology class maps $\mathcal{G}_i$ onto $S_i \cap H_1(M,\Z)$. It is injective by the genericity of $(M,m)$. This finishes the proof of the lemma.
\qed 

In the following result, we refer the reader to the introduction for the definitions of  $\mathcal{B}_1$, $ \mathcal{F}$, and $\Omega$.

\begin{theorem}
Let $(M,m)$ be a giraffe of genus $g$. Furthermore assume that $(M,m)$ is generic in the sense of Definition \ref{def:genericity}.  
Then we have \begin{itemize}[itemsep=2ex,leftmargin=0.5cm]
\item 
$$ 
\lim_{T \rightarrow \infty} \frac{N(T)}{T^2}= \sum_{i=1}^g \Omega (S_i \cap \mathcal{B}_1 )
$$
\item
$$
\bigcup_{i=1}^g S_i \cap \mathcal{B}_1= \mathcal{F}
$$
\item
$$
\Omega (S_i \cap \mathcal{B}_1 )=\Omega (\mathcal{F}).
$$
\end{itemize}
\end{theorem}
So a generic giraffe of genus $g$ verifies Conjecture \ref{conj}.

\proof
By Lemma \ref{lem:longthinnecks} we have 
\begin{equation}\label{G union G i}
\mathcal{G} = \bigcup_{i=1}^g \mathcal{G}_i
\end{equation}
so, denoting 
$$
N_i(T) = \sharp \{  \gamma \in \mathcal{G}_i \co l_m (\gamma) \leq T \},
$$
we have 
$$
N(T) = \sum_{i=1}^g N_i(T).
$$
On the other hand, by Lemma \ref{lem G i et S i}, for any $T >0$, 
\begin{eqnarray*}
N_i(T) &=& \sharp \{ h \in S_i \cap H_1(M,\Z) \co \|h\| \leq T \}\\
&=&  \sharp T S_i \cap \mathcal{B}_1.
\end{eqnarray*}
Since the set $S_i \cap \mathcal{B}_1$ is convex and the lattice $S_i \cap H_1(M,\Z)$ has determinant $1$ (because the subspace $S_i$ is symplectic), the classical Minkovski theorem says that 
$$ 
\lim_{T \rightarrow \infty} \frac{N_i(T)}{T^2}=  \Omega (S_i \cap \mathcal{B}_1 )
$$
whence, by (\ref{G union G i}),
$$ 
\lim_{T \rightarrow \infty} \frac{N(T)}{T^2}= \sum_{i=1}^g \Omega (S_i \cap \mathcal{B}_1 )
$$
which is the first statement of the theorem.

Now, by (\ref{G union G i}), $\mathcal{F}$ is the closure in $H_1(M,\R)$ of the set
$$
\bigcup_{i=1}^g \{ t \frac{\left[\gamma \right]}{\|\left[\gamma \right]\|} \co \gamma \in \mathcal{G}_i, \  t \in \left[0,1\right] \}
$$
whence, by Lemma \ref{lem G i et S i}, $\mathcal{F}$ is the closure in $H_1(M,\R)$ of the set
$$
\bigcup_{i=1}^g \{ t \frac{h}{\|h\|} \co h \in S_i \cap H_1(M,\Z), \  t \in \left[0,1\right] \}
$$
that is, 
$$
\mathcal{F}= \bigcup_{i=1}^g S_i \cap \mathcal{B}_1
$$
which is the second statement of the theorem.

Finally, since $S_i \cap S_j = \{ 0 \}$ for any $i \neq j$, we have 
$$
\Omega ( \bigcup_{i=1}^g S_i \cap \mathcal{B}_1 ) = \sum_{i=1}^g \Omega ( S_i \cap \mathcal{B}_1 ) 
$$
which is the third statement of the theorem.
\qed

In the general case - that is, when $(M,m)$ is not a giraffe - there is no reason why homology classes of homologically minimal geodesics should be distributed into finitely many lattices of rank $2$. So in the remainder of this paper what we do is look for a suitable replacement for the subspaces $S_i$, which leads us into the next subsection. 

\subsection{More notation: definition of $V(\Gamma)$ and $L(\Gamma)$}

Let
 \begin{itemize}[itemsep=2ex,leftmargin=0.5cm]
 \item $M$ be a closed, orientable, surface of genus $g$, equipped with a generic Finsler metric $m$
 \item $\gamma_1, \ldots \gamma_{g-1}$ be simple closed geodesics such that the formal sum $\Gamma := \gamma_1 + \ldots + \gamma_{g-1}$ is a minimizing multicurve
 \item $V(\Gamma)$ be the vector subspace of $H_1(M,\R)$ generated by the homology classes $\left[\gamma_1\right], \ldots \left[\gamma_{g-1}\right]$.
\end{itemize}
Recall that 
$$V(\Gamma)^{\perp} = \bigcap_{i=1}^{g-1} \left[\gamma_i\right]^{\perp},
$$
where orthogonality is meant with respect to the symplectic intersection form. The subspace $V(\Gamma)$ is integer, that is, it is generated by integer homology classes (elements of $H_1(M,\Z)$). Furthermore, the intersection form is integer, that is, for any $h,h'$ in $H_1(M,\Z)$, we have $\inter (h,h') \in \Z$.

Therefore, in coordinates relative to an integer basis of $H_1(M,\R)$ (that is, a basis of $H_1(M,\R)$ that consists of elements of $H_1(M,\Z)$), the equation
$$\inter( \left[\gamma_i\right], . )=0$$
is a system of linear equations with integer coefficients. Thus the vector subspace of its solutions is generated by integer vectors. Hence the subspace $V(\Gamma)^{\perp}$ is integer. 

Now observe that since $\Gamma$ is a minimizing multicurve, its connected components are pairwise disjoint, so 
$\inter ( \left[\gamma_i\right], \left[\gamma_j\right],)=0$ for all $i,j = 1,\ldots g-1$. This means that 
$$
V(\Gamma) \subset V(\Gamma)^{\perp}
$$
thus the quotient space $V(\Gamma)^{\perp}/ V(\Gamma) $ is well defined. Since both $V(\Gamma)$ and $V(\Gamma)^{\perp}$ are integer subspaces of $H_1(M,\R)$, the quotient group 
$$
L(\Gamma) := \left( V(\Gamma)^{\perp} \cap H_1(M,\Z) \right) / \left( V(\Gamma) \cap H_1(M,\Z) \right)
$$
is a lattice (i.e. a discrete, cocompact additive subgroup) in $V(\Gamma)^{\perp}/ V(\Gamma) $.

We want to prove that the intersection form quotients to a symplectic form on $V(\Gamma)^{\perp}/ V(\Gamma)$.
The kernel of the restriction of the intersection form to $V(\Gamma)^{\perp}$ is 
$$\left( V(\Gamma)^{\perp} \right)^{\perp} = V(\Gamma).
$$
 Thus the intersection form is well defined, and non-degenerate, hence symplectic, on the quotient space. 

Finally, we point out that since $(M,m)$ is generic, the homology classes $\left[ \gamma_1 \right], \ldots \left[ \gamma_{g-1} \right]$ are linearly independent, so $\dim V(\Gamma) = g-1$, thus $\dim V(\Gamma)^{\perp} = 2g-(g-1)=g+1$. Hence 
$$
\dim V(\Gamma)^{\perp}/ V(\Gamma) = g+1-(g-1) = 2.
$$
In short, we have proved that $ V(\Gamma)^{\perp}/ V(\Gamma)$ has dimension two, and comes with a symplectic form and a lattice $L(\Gamma)$, which are quotients of, respectively, the intersection form and the integer lattice.

\begin{lemma}\label{lemme nu}
Let
 \begin{itemize}[itemsep=2ex,leftmargin=0.5cm]
 \item $M$ be a closed, orientable, surface of genus $g$, equipped with a generic Finsler metric $m$,
 \item $\gamma_1, \ldots \gamma_{k}$ be simple closed geodesics, such that the formal sum $\Gamma := \gamma_1 + \ldots + \gamma_{k}$ is a minimizing multicurve,
 \item $h$ be a homology class in $ V(\Gamma)^{\perp}$.
 \end{itemize}
 Then there exists a minimizing measure $\nu$, such that $\left[\nu\right]= h \mod V(\Gamma)$, $\nu+\mu_{\Gamma}$ is a minimizing measure, and $\spt (\nu) \cap (\Gamma, \dot\Gamma) = \emptyset$, wfhere $(\Gamma, \dot\Gamma) $ denotes the set of tangent vectors to $\Gamma$ in the unit tangent bundle $T^1M$ .
\end{lemma}
\proof
By Theorem \ref{thm-gafa} there exists $s(h_0,h)>0$ such that the subset of the unit sphere $\partial \mathcal{B}_1$
$$
 \left\{\frac{ h_0+ sh}{|| h_0+sh||} \, \co s \in \left[0,s(h_0,h)\right] \right\}
$$
is a straight segment. Let $\mu$ be an $h_0 +s h$-minimizing measure, for some $0< s \leq s(h_0,h)$.

 For each $i=1,\ldots g-1$ let $V_i$ be the neighborhood of $(\gamma_i,\dot\gamma_i)$ given by Proposition \ref{two-sided}. Let $V$ be the union over $i=1,\ldots g-1$ of the $V_i$. First let us prove that $V \cap \spt (\mu )$ is $\phi_t$-invariant. Indeed by Proposition \ref{two-sided} a minimizing geodesic that enters $V$ either
\begin{itemize}[itemsep=2ex,leftmargin=0.5cm]
 \item is asymptotic to one of the $\gamma_i$,
 \item is homotopic to one of the $\gamma_i$,
 \item or cuts one of the $\gamma_i$ with constant sign.
\end{itemize}
In the second case, the minimizing geodesic cannot be in the support of any minimizing measure by Lemma \ref{asymptote_fermee}.
In the third case, the minimizing geodesic cannot be in the support of a minimizing measure $\tau$ such that $\inter (\left[\tau\right], \gamma_i)=0$ for $i=1,\ldots g-1$, in particular it cannot be in the support of an $h_0 +s h$-minimizing measure, because of the Graph Property \cite{Mather91} which says that the canonical projection $p \co TM \rightarrow M$, restricted to the support of a minimizing measure, is injective, and its inverse is Lipschitz. 


Therefore $V \cap \spt (\mu )$ consists of periodic orbits homotopic to some or all of the $\gamma_i$. Thus it is $\phi_t$-invariant. 

For any measurable subset $A$ of $T^{1}M$, set
\begin{eqnarray*}
\alpha (A) & := & \mu(A\cap V) \\
\beta(A) & := & \mu(A \setminus V).
\end{eqnarray*}

Then $\alpha$ and $\beta$ are two measures on $T^{1}M$. They are invariant by the geodesic flow because $V\cap \spt (\mu )$, as well as its complement in $ \spt (\mu )$, is $\phi_t$-invariant. They are both minimizing because their supports are contained in the support of a minimizing measure  \cite{Mane92}. The support of $\nu$ is disjoint from $(\Gamma, \dot\Gamma)$ by the definition of $\beta$.

Since the support of $\alpha$ consists of periodic orbits homotopic to some or all of the $\gamma_i$, the homology class of $\alpha$ is contained in $V(\Gamma)$. We have $\alpha + \beta = \mu$, hence $\left[\alpha\right]+\left[\beta\right]= h_0 +s h$.
Thus 
$$
\nu:= \frac{1}{s}\beta
$$
is a minimizing measure by Lemma \ref{lemme combinaisons lineaires}. Its homology is 
$$
h + \frac{1}{s}\left( h_0 - \left[\alpha\right] \right),
$$
 thus $\left[\nu\right]= h \mod V(\Gamma)$.
 
 Moreover, since all of the $\gamma_i$ are contained in the support of $\alpha$,$\alpha$ cannot be expressed as a convex combination of a proper subset of $\mu_1, \ldots \mu_{g-1}$, so  there exist $a_1, \ldots a_{g-1} > 0$ such that
 $$ 
 \alpha = \sum_{i=1}^{g-1} a_i \mu_i.
 $$
 Now 
 $$
 \alpha + \beta = \sum_{i=1}^{g-1} a_i\mu_i + \beta
 $$
 is a minimizing measure, so by Lemma \ref{lemme combinaisons lineaires}, 
 $$
 \nu+\mu_{\Gamma} = \sum_{i=1}^{g-1} \mu_i + \frac{1}{s} \beta
 $$
 is also a minimizing measure.
\qed
\subsection{Still more notation: definitions of $\mathcal{G}(\Gamma)$ and $\mathcal{F}(\Gamma)$}

Let
 \begin{itemize}[itemsep=2ex,leftmargin=0.5cm]
 \item $M$ be a closed, orientable, surface of genus $g$, equipped with a generic Finsler metric $m$
 \item $\gamma_1, \ldots \gamma_{g-1}$ be simple closed geodesics, such that the formal sum $\Gamma := \gamma_1 + \ldots + \gamma_{g-1}$ is a minimizing multicurve.
 \end{itemize}

We denote by
 \begin{itemize}[itemsep=2ex,leftmargin=0.5cm]
 \item $\mathcal{G}(\Gamma)$ the set of closed geodesics $\alpha$, disjoint from $\Gamma$, such that $\Gamma + \alpha$ is a minimizing multicurve (observe that if $\alpha \in \mathcal{G}(\Gamma)$, then $n\alpha \in \mathcal{G}(\Gamma)$ for any $n \in \Z$)
 \item $\mu_{i}$ the element of $\mathcal{M}$ which is supported on $(\gamma_i, \dot \gamma_i) \subset T^1 M$
 \item $\mu_{\Gamma} := \mu_1 + \ldots +\mu_{g-1}$ 
 \item $\mathcal{F}(\Gamma)$ the set of homology classes $h$ of measures $\mu$, supported outside of $(\Gamma, \dot\Gamma)$, such that $\mu_{\Gamma} + \mu $ is a 
 minimizing measure
 \item by $\mathcal{F}_t(\Gamma)$ the intersection of $\mathcal{F}(\Gamma)$ with the ball of the stable norm of radius $t$, centered at the origin. 
 \end{itemize}
Observe that 
\begin{itemize}[itemsep=2ex,leftmargin=0.5cm]
 \item $\mathcal{M}$ is invariant under sums and scalar multiplication by a nonnegative number, that is, $\mathcal{M}$ is a convex cone with vertex at the zero measure. 
 \item $\mathcal{F}(\Gamma)$ is invariant under scalar multiplication by a nonnegative number, that is, $\mathcal{F}(\Gamma)$ is a cone with vertex at the origin. 
 \item for any $\alpha$ in $\mathcal{G}(\Gamma)$, the homology class of $\alpha$ lies in $\mathcal{F}(\Gamma)$ 
 \item for any $\alpha$ in $\mathcal{G}(\Gamma)$, $\left[\alpha\right] \mod V(\Gamma) \in L(\Gamma)$
 \item for any $\mu$ in $\mathcal{M}$ such that $\left[\mu\right] \in \mathcal{F}(\Gamma)$, we have $\inter ( \left[\mu\right], \left[ \gamma_i\right]) = 0$ for $i=1,\ldots g-1$. Therefore 
 $\mathcal{F}(\Gamma) \subset V(\Gamma)^{\perp} $.
\end{itemize}
\section{}
\begin{proposition}\label{injection F Gamma}
Let 
\begin{itemize}[itemsep=2ex,leftmargin=0.5cm]
 \item $M$ be a closed, orientable, surface of genus $g$, equipped with a generic Finsler metric $m$
 \item $\gamma_1, \ldots \gamma_{g-1}$ be simple closed geodesics such that the formal sum $\Gamma := \gamma_1 + \ldots + \gamma_{g-1}$ is a minimizing multicurve.
\end{itemize}
Then the canonical projection 
$$
 \mathcal{F}(\Gamma)\longrightarrow V(\Gamma)^{\perp}/ V(\Gamma)
 $$
is a bijection.
\end{proposition}
\proof
Let $h$ be an element of $V(\Gamma)^{\perp}/ V(\Gamma)$. 
By Lemma \ref{lemme nu} there exists a minimizing measure $\nu$, such that $\left[\nu\right] \mod V(\Gamma) = h $, 
$\spt (\nu) \cap (\Gamma, \dot\Gamma ) = \emptyset$, and $ \nu+\mu_{\Gamma}$ is a minimizing measure.
Then the homology class of $\nu$, which is $ h \mod V(\Gamma)$, is contained in $\mathcal{F}(\Gamma)$, which proves that the canonical projection 
$$\mathcal{F}(\Gamma)\longrightarrow V(\Gamma)^{\perp}/ V(\Gamma)$$
is onto. Now let us prove that it is one-to-one.

Take $h,h'$ in $\mathcal{F}(\Gamma)$ such that $h=h' \mod V(\Gamma)$, and measures $\mu, \mu'$ in $\mathcal{M}$ such that 
\begin{itemize}[itemsep=2ex,leftmargin=0.5cm]
 \item $\left[\mu\right]=h$ and $\left[\mu'\right]=h'$
 \item $\spt (\mu)\cap (\Gamma, \dot\Gamma)=\spt (\mu')\cap (\Gamma, \dot\Gamma)= \emptyset $
 \item $\mu + \mu_{\Gamma}= \mu + \mu_1+\ldots \mu_{g-1}$ and $\mu' + \mu_{\Gamma}= \mu' + \mu_1+\ldots \mu_{g-1}$ are minimizing.
\end{itemize} 
We start by proving that $\mu +\mu' + \mu_{\Gamma}$ is minimizing. Since $h=h' \mod V(\Gamma)$, there exist real numbers $\lambda_1, \ldots \lambda_{g-1}$ such that 
$$
h'= h + \lambda_1\left[\gamma_1\right]+ \ldots +\lambda_{g-1}\left[\gamma_{g-1}\right] .
$$
Set, for any $\lambda \in \left] 0,1 \right[ $, 
$$
h_{\lambda} := \lambda \frac{h_0}{\|h_0\|} + (1-\lambda) \frac{h'}{\|h'\|}.
$$
We have
\begin{eqnarray*}
\lambda \frac{h_0}{\|h_0\|} + (1-\lambda) \frac{h'}{\|h'\|}&=&
 \frac{\lambda}{\|h_0\|} \sum_{i=1}^{g-1} \left[\gamma_i \right] + \frac{1-\lambda}{\|h'\|}\left( h+ \sum_{i=1}^{g-1}\lambda_i \left[\gamma_i \right] \right) \\
&=& \sum_{i=1}^{g-1}\left( \frac{\lambda}{\|h_0\|} + \frac{(1-\lambda)\lambda_i }{\|h'\|} \right)\left[\gamma_i \right]+ \frac{1-\lambda}{\|h'\|} h
\end{eqnarray*}
so if we take $\lambda$ sufficiently close to $1$, $h_{\lambda}$ is a linear combination, with positive coefficients, of 
$h ,\left[\gamma_1\right], \ldots ,\left[\gamma_{g-1}\right] $. Thus by Lemma \ref{lemme combinaisons lineaires}, since these coefficients are positive, 
the measures
\begin{eqnarray*}
\nu' & := & \lambda \frac{\mu_{\Gamma}}{\|h_0\|} + (1-\lambda) \frac{\mu'}{\|h'\|} \\
\nu & := & \sum_{i=1}^{g-1}\left( \frac{\lambda}{\|h_0\|} + \frac{(1-\lambda)\lambda_i }{\|h'\|} \right)\mu_i + \frac{1-\lambda}{\|h'\|} \mu
\end{eqnarray*}
are both minimizing, and since their homology class is $h_{\lambda} $, they are $h_{\lambda} $-minimizing. Therefore $\nu + \nu'$ is $2h_{\lambda} $-minimizing. Using Lemma \ref{lemme combinaisons lineaires} again, we deduce that $\mu +\mu' + \mu_{\Gamma}$ is minimizing.

Hence, by the Graph Property, $\mu+\mu'$ may be viewed as an invariant measure of a Lipschitz flow on $M\setminus \Gamma$, which is homeomorphic to a torus with $g-1$ punctures, so $\mu+\mu'$ may be viewed as an invariant measure of a Lipschitz flow on $\T^2$. Recall from \cite{Katok} that an invariant measure of a Lipschitz flow on $\T^2$ is either ergodic, or supported on periodic orbits (or both, if it is supported on one periodic orbit). 

\textbf{First case}: $\mu+\mu'$ is ergodic. Then either there exists $a \in \R$ such that $\mu= a \mu'$, or there exists $a \in \R$ such that $\mu'= a \mu$.
So, either there exists $a \in \R$ such that $h= a h'$, or there exists $a \in \R$ such that $h'= a h$. By the hypothesis that $h=h' \mod V(\Gamma)$, this entails that either $a=1$, or $h=h'=0 \mod V(\Gamma)$.

If $a=1$, we have $h=h'$, which proves that the canonical projection 
$$\mathcal{F}(\Gamma)\longrightarrow V(\Gamma)^{\perp}/ V(\Gamma)$$
is one-to-one, and thus the proposition.

If $h=h'=0 \mod V(\Gamma)$, then there exist $a_1, \ldots , a_{g-1}, b_1, \ldots , b_{g-1}$ in $\R$ such that 
$$
h= \sum_{i=1}^{g-1}a_i h_i \mbox{ and } h'= \sum_{i=1}^{g-1}b_i h_i.
$$
Let us take $\lambda >0$ such that 
$$
1- \lambda a_i >0 \mbox{ and } 1- \lambda b_i >0 \ \forall i=1, \ldots , g-1.
$$
Recall that the measures $\mu + \mu_{\Gamma}= \mu + \mu_1+\ldots \mu_{g-1}$ and $\mu' + \mu_{\Gamma}= \mu' + \mu_1+\ldots \mu_{g-1}$ are minimizing. So by Lemma \ref{lemme combinaisons lineaires}, 
\begin{eqnarray*}
\tau & := & \lambda\mu +\sum_{i=1}^{g-1}(1-\lambda a_i) \mu_i\\
\tau' & := & \lambda\mu' +\sum_{i=1}^{g-1}(1-\lambda b_i) \mu_i
\end{eqnarray*}
are also minimizing. Now $\left[\tau\right]= \left[\tau'\right] = h_0$, which is rational, hence by genericity $\tau= \tau'$. Let us take, for each $ i=1, \ldots , g-1$, a neighborhood $V_i$ of $(\gamma_i, \dot\gamma_i)$ in $T^1 M$, such that \begin{itemize}
 \item $V_i \cap \spt (\mu) = V_i \cap \spt (\mu') = \emptyset$
 \item $V_i \cap V_j = \emptyset$ if $i \neq j$.
 \end{itemize}
Then we have, for each $ i=1, \ldots , g-1$, 
\begin{eqnarray*}
\tau (V_i) & = & (1-\lambda a_i) \mu_i (V_i) = (1-\lambda a_i) \mbox{length}(\gamma_i) \\
\tau' (V_i) & = & (1-\lambda b_i) \mu_i (V_i) = (1-\lambda b_i) \mbox{length}(\gamma_i), 
\end{eqnarray*}
which, since $\tau= \tau'$, implies $a_i=b_i \ \forall i=1, \ldots , g-1$, whence $h=h'$. So the proposition is proved in the case when $\mu+\mu'$ is ergodic.

\textbf{Second case}: $\mu+\mu'$ is supported on periodic orbits. Recall that two simple closed curves on a torus which do not intersect have homology classes which are proportional, that is, one is a multiple of the other. Also recall that the homology of $M \setminus \Gamma$ is isomorphic to $H_1 (M,\R) / V(\Gamma)$. Thus there exist $h''$ in $H_1 (M,\Z)$ such that 
$$
\left[\mu\right], \left[\mu'\right] \in V(\Gamma)\oplus \R h''
$$
so there exist $a_1, \ldots , a_{g}$ in $\R$ such that 
$$
h=\left[\mu\right]= \sum_{i=1}^{g-1}a_i h_i + a_g h''.
$$
Since $h=h' \mod V(\Gamma)$, there exist $b_1, \ldots , b_{g-1}$ in $\R$ such that 
$$
h'=\left[\mu'\right]= \sum_{i=1}^{g-1}b_i h_i + a_g h''.
$$
Let us take $\lambda >0$ such that 
$$
1- \lambda a_i >0,\ 1- \lambda b_i >0 \ \forall i=1, \ldots , g-1 \mbox{ and } \lambda a_g \in \Q.
$$
Then 
\begin{eqnarray*}
\tau & := & \lambda\mu +\sum_{i=1}^{g-1}(1-\lambda a_i) \mu_i\\
\tau' & := & \lambda\mu' +\sum_{i=1}^{g-1}(1-\lambda b_i) \mu_i
\end{eqnarray*}
are both minimizing measures, by Lemma \ref{lemme combinaisons lineaires}. But 
$$
\left[\tau\right]= \left[\tau'\right] = h_0 + a_g h''
$$
which is a rational homology class. Hence by genericity $\tau= \tau'$. As in the first case, we then show that $a_i=b_i \ \forall i=1, \ldots , g-1$, whence $h=h'$. This finishes the proof of Proposition \ref{injection F Gamma}.
\qed
\begin{lemma}\label{ F Gamma ferm�}
Let 
\begin{itemize}[itemsep=2ex,leftmargin=0.5cm]
 \item $M$ be a closed, orientable, surface of genus $g$, equipped with a generic Finsler metric $m$
 \item $\gamma_1, \ldots \gamma_{g-1}$ be simple closed geodesics such that the formal sum $\Gamma := \gamma_1 + \ldots + \gamma_{g-1}$ is a minimizing multicurve.
\end{itemize}
Then $ \mathcal{F}(\Gamma)$ is closed in $H_1 (M,\R)$. 
\end{lemma}
\proof
Take
\begin{itemize}[itemsep=2ex,leftmargin=0.5cm]
 \item a sequence $h_n$ of elements of $\mathcal{F}(\Gamma)$, that converges to some $h$ in $H_1(M,\R)$
 \item for each $n$ in $\N$, an $h_n$-minimizing measure $\mu_n$ such that $\mu_n + \mu_{\Gamma}$ is minimizing, and $\spt (\mu_n) \cap (\Gamma, \dot\Gamma) = \emptyset$.
\end{itemize}
Then, if $\mu$ is any limit point of the sequence $\mu_n$ in the weak$^*$ topology, $\mu$ is $h$-minimizing, and $\mu + \mu_{\Gamma}$ is minimizing because a limit of minimizing measure is minimizing. To prove that $h \in \mathcal{F}(\Gamma)$, it only remains to prove that $\spt (\mu) \cap (\Gamma, \dot\Gamma) = \emptyset$. By Proposition \ref{two-sided} and Lemma \ref{asymptote_fermee}, there exists a neighborhood $V$ of $(\Gamma, \dot\Gamma)$ in $T^1 M$ such that for all $n$ in $\N$, $\spt (\mu_n) \cap V = \emptyset$. Then for any continuous function $f$ supported inside $V$, we have $\int f d \mu_n = 0$, hence $\int f d \mu = 0$, so $\mu$ is supported outside $V$. Thus $\spt (\mu) \cap (\Gamma, \dot\Gamma) = \emptyset$, which proves that $h \in \mathcal{F}(\Gamma)$. Therefore $\mathcal{F}(\Gamma)$ is closed in $H_1 (M,\R)$. 
\begin{lemma}\label{homeo F Gamma}
Let 
\begin{itemize}[itemsep=2ex,leftmargin=0.5cm]
 \item $M$ be a closed, orientable, surface of genus $g$, equipped with a generic Finsler metric $m$
 \item $\gamma_1, \ldots \gamma_{g-1}$ be simple closed geodesics such that the formal sum $\Gamma := \gamma_1 + \ldots + \gamma_{g-1}$ is a minimizing multicurve.
\end{itemize}
Then the canonical projection 
$$
 \mathcal{F}(\Gamma)\longrightarrow V(\Gamma)^{\perp}/ V(\Gamma)
 $$
is a homeomorphism. 
\end{lemma}
\proof
The continuity of the canonical projection is obvious, so all we have to prove is the continuity of the inverse map. Take a sequence $h_n$ of points in $V(\Gamma)^{\perp}/ V(\Gamma)$, which converges to some $h \in V(\Gamma)^{\perp}/ V(\Gamma)$. By Proposition \ref{injection F Gamma}, there exist elements $h', h'_n, n\in \N$ of $ \mathcal{F}(\Gamma)$ such that $h'_n \mod V(\Gamma) = h_n$ for all $n\in \N$, and $h' \mod V(\Gamma) = h$.
What we need to prove is that $h'_n$ converges to $h'$.

First let us prove, by contradiction, that the sequence $h'_n$ is bounded in $H_1 (M,\R)$. Assume, after possibly taking a subsequence, that $h'_n$ goes to infinity. Consider the sequence 
$$h''_n :=h'_n / \|h'_n\|.$$
Then $h''_n \in \mathcal{F}(\Gamma) $ by homogeneity of $\mathcal{F}(\Gamma)$. On the other hand, for all $n$, $h''_n $ lies on the unit sphere of the stable norm, which is compact, so $h''_n$ has a limit point $h''$ such that $\|h''\|=1$. Since $ \mathcal{F}(\Gamma)$ is closed in $H_1(M,\R)$ by Lemma \ref{ F Gamma ferm�}, $h''$ lies in $ \mathcal{F}(\Gamma)$. Now the projection of $h''_n$ to $V(\Gamma)^{\perp}/ V(\Gamma)$
is $h_n / \|h'_n\|$, which converges to zero. By continuity of the projection, it follows that $h''$ projects to $0 \in V(\Gamma)^{\perp}/ V(\Gamma)$. But $0 \in H_1 (M,\R)$ lies in $ \mathcal{F}(\Gamma)$, and projects to $0 \in V(\Gamma)^{\perp}/ V(\Gamma)$. By the injectivity of the projection (Proposition \ref{injection F Gamma}), this entails $h''=0$, which contradicts $\|h''\|=1$. This contradiction shows that the sequence $h'_n$ is bounded in 
$H_1 (M,\R)$.

Therefore it has limit points in $H_1(M,\R)$. Any such limit point lies in $ \mathcal{F}(\Gamma)$ by Lemma \ref{ F Gamma ferm�}. By continuity of the projection, any limit point of the sequence $h'_n$ projects to $h$. By the injectivity of the projection (Proposition \ref{injection F Gamma}), the sequence $h'_n$ then converges to $h'$, which proves the lemma.
\qed

\begin{lemma}\label{bord negligeable}
Let $M$ be a closed, orientable, surface of genus $g$, equipped with a generic Finsler metric $m$ and let $\gamma_1, \ldots \gamma_{g-1}$ be simple closed geodesics such that the formal sum $\Gamma := \gamma_1 + \ldots + \gamma_{g-1}$ is a minimizing multicurve.

Then the canonical projection $\mathcal{P}_1(\Gamma)$ of $\mathcal{F}_1(\Gamma)$ to $V(\Gamma)^{\perp}/ V(\Gamma) $ is a compact subset of $V(\Gamma)^{\perp}/ V(\Gamma) $, whose boundary is Lebesgue-negligible.
\end{lemma}
\proof
Since $\mathcal{F}_1(\Gamma)$ is the intersection of $\mathcal{F}(\Gamma)$, which is closed in $H_1(M,\R)$ by Lemma \ref{ F Gamma ferm�}, with the unit ball of the stable norm, which is compact, it turns out that $\mathcal{F}_1(\Gamma)$ is compact, and so is its projection to $V(\Gamma)^{\perp}/ V(\Gamma) $.

Now we prove that the boundary $\partial \mathcal{P}_1(\Gamma)$ of $\mathcal{P}_1(\Gamma)$ is Lebesgue-negligible.
For this we show that for any $h$ in $V(\Gamma)^{\perp}/ V(\Gamma)\setminus \{0\} $, there exists a unique $\lambda > 0 $ such that $\lambda h$ lies on $\partial \mathcal{P}_1(\Gamma)$. The Lebesgue-negligibility of $\partial \mathcal{P}_1(\Gamma)$ then follows from Fubini's theorem. Take 
\begin{itemize}[itemsep=2ex,leftmargin=0.5cm]
 \item $h$ in $V(\Gamma)^{\perp}/ V(\Gamma)\setminus \{0\} $
 \item $h'$ in $\mathcal{F}(\Gamma)$, given by Proposition \ref{injection F Gamma}, such that $h' \mod V(\Gamma) = h$.
\end{itemize}
Then, for any $\lambda >0$, $\lambda h'$ lies in $\mathcal{F}(\Gamma)$, and projects to $\lambda h$. By Proposition \ref{injection F Gamma}, 
$\lambda h'$ is the unique element of $\mathcal{F}(\Gamma)$ which projects to $\lambda h$. Now, for $\lambda \leq \frac{1}{\|h'\|}$, we have $\lambda h' \in \mathcal{F}_1(\Gamma)$, so $\lambda h \in \mathcal{P}_1(\Gamma)$. For $\lambda > \frac{1}{\|h'\|}$, we have $\lambda h' \not\in \mathcal{F}_1(\Gamma)$, and since $\lambda h'$ is the unique element of $\mathcal{F}(\Gamma)$ which projects to $\lambda h$, it turns out that $\lambda h \not\in \mathcal{P}_1(\Gamma)$. Thus the point $\frac{1}{\|h'\|}h$ lies on $\partial \mathcal{P}_1(\Gamma)$, because it can be approximated both from inside and outside $ \mathcal{P}_1(\Gamma)$. On the other hand, for any $\lambda < \frac{1}{\|h'\|}$, if $h_n$ is any sequence of elements of $V(\Gamma)^{\perp}/ V(\Gamma)$ which converges to $\lambda h$, denoting $h'_n$ the only point of $\mathcal{F}(\Gamma)$ such that $h'_n \mod V(\Gamma) = h_n$, by Lemma \ref{homeo F Gamma}, $h'_n$ converges to $\lambda h'$.

Recall that $\lambda < \frac{1}{\|h'\|}$, so $\|\lambda h'\| <1$, hence for $n$ large enough we have $\|h'_n\|<1$. Thus, for $n$ large enough, $h_n \in \mathcal{P}_1(\Gamma)$, that is, $\lambda h$ lies in the interior of $\mathcal{P}_1(\Gamma)$ as a subset of $V(\Gamma)^{\perp}/ V(\Gamma)$.

Therefore $\lambda = \frac{1}{\|h'\|}$ is the unique $\lambda > 0 $ such that $\lambda h$ lies on $\partial \mathcal{P}_1(\Gamma)$. This finishes the proof of the lemma. 
\qed


\begin{lemma}\label{injection G Gamma}
Let $M$ be a closed, orientable, surface of genus $g$, equipped with a generic Finsler metric $m$ and let $\gamma_1, \ldots \gamma_{g-1}$ be simple closed geodesics such that the formal sum $\Gamma := \gamma_1 + \ldots + \gamma_{g-1}$ is a minimizing multicurve.

Then the map 
$$
\begin{array}{rcl}
\psi_{\Gamma} \co \mathcal{G}(\Gamma) & \longrightarrow & L(\Gamma) \\
\alpha & \longmapsto & \left[ \alpha \right] \mod V(\Gamma)
\end{array}
$$
is a bijection.
\end{lemma}
\proof
First let us observe that the map $\psi_{\Gamma}$ is well defined : for any $\alpha$ in $\mathcal{G}(\Gamma)$, we have
$$
\left[\alpha\right] \in \left( \mathcal{F}(\Gamma) \cap H_1 (M,\Z) \right) \ \subset \ \left( V(\Gamma)^{\perp} \cap H_1 (M,\Z)\right) 
$$
so 
$$
 \left[ \alpha \right] \mod V(\Gamma) \in \left( V(\Gamma)^{\perp} \cap H_1 (M,\Z)/ V(\Gamma) \right) = L(\Gamma).
 $$

Now let us prove that $\psi_{\Gamma}$ is injective. Take $\alpha, \alpha'$ in $\mathcal{G}(\Gamma)$ such that $ \left[ \alpha \right] = \left[ \alpha' \right] \mod V(\Gamma)$. Then, since $\left[\alpha\right], \left[\alpha'\right]$ in $\mathcal{F}(\Gamma)$, Proposition \ref{injection F Gamma} says that $\left[\alpha\right] = \left[\alpha'\right]$. Then $\alpha= \alpha'$ because the metric $m$ is generic. Therefore $\psi_{\Gamma}$ is injective.

It remains to prove that $\psi_{\Gamma}$ is onto. Take $l$ in $L(\Gamma)$. By Proposition \ref{injection F Gamma}, there exists $h$ in $\mathcal{F}(\Gamma)$ such that $h \mod V(\Gamma) =l$. Hence, by the definition of $L(\Gamma)$, there exists $l_1 \in H_1(M,\Z) \cap V(\Gamma)^{\perp}$, and $a_1,\ldots, a_{g-1}$ in $\R$ such that $h = l_1 + a_1 h_1, + \ldots + a_{g-1}h_{g-1}$. 

Since $h \in \mathcal{F}(\Gamma)$, we may take an $h$-minimizing measure $\mu$, supported away from $(\Gamma, \dot\Gamma)$, such that $\mu + \mu_{\Gamma}$ is minimizing. Then by Lemma \ref{lemme combinaisons lineaires}, for any integer $n$ such that $n \geq a_i$, $i=1, \ldots g-1$, the measure 
$$
\nu := \mu + (n- a_{1})\mu_1 +\ldots +(n- a_{g-1})\mu_{g-1}
 $$
 is minimizing. But its homology class is $l_1 + h_1 + \ldots h_{g-1}$, which is integer, so by Proposition \ref{rational} $\nu$ is supported on closed geodesics. Thus $\mu$ is supported on closed geodesics. Take two closed geodesics $\gamma$ and $\gamma'$ in the support of $\mu$. Then $\gamma$ and $\gamma'$ are simple, non-separating closed curves on the punctured torus $M\setminus \Gamma$, and by the Graph Theorem they are either disjoint or equal, so their homology classes (as curves in $\T^2$) are proportional. Thus there exist an integer homology class $h_g \in H_1(M,Z)$, and integers  $a_g, b_g$, such that 
\begin{eqnarray*}
\left[\gamma\right] &=& \sum_{i=1}^{g-1}a_i h_i + a_g h_g \\
\left[\gamma'\right] &=& \sum_{i=1}^{g-1}b_i h_i + b_g h_g 
\end{eqnarray*}
Reversing, if necessary, the orientations of $\gamma$ and $\gamma'$, we may assume that $a_g, b_g \geq 0$. Take $n, n' \in \N$ such that $na_g=n'b_g$.
Let \begin{itemize}[itemsep=2ex,leftmargin=0.5cm]
 \item $\tau$ be the measure in $\mathcal{M}$ supported by $(\gamma, \dot\gamma)$, whose homology class is $\left[\gamma\right] $
 \item $\tau'$ be the measure in $\mathcal{M}$ supported by $(\gamma', \dot\gamma')$, whose homology class is $\left[\gamma'\right] $.
\end{itemize}
Since $\mu + \mu_{\Gamma}$ is minimizing, and $(\gamma, \dot\gamma)$ is contained in the support of $\mu$, $\tau+ \mu_{\Gamma}$ is minimizing, and so is
$$
\tau_1 := n \nu + \sum_{i=1}^{g-1} \max (n'b_i - n a_i, 0) \mu_i
$$
by Lemma \ref{lemme combinaisons lineaires}.
Likewise, 
$$
\tau_2 := n' \nu' + \sum_{i=1}^{g-1} \max (n a_i-n'b_i , 0) \mu_i
$$
is minimizing. But 
$$
\left[\tau_1\right]= \left[\tau_2\right]= n a_g h_g + \sum_{i=1}^{g-1}\left( \max (n'b_i , n a_i \right) h_i
$$
which is an integer homology class. Therefore, since the metric $m$ is generic, we have $\tau_1=\tau_2$, whence $\tau= \tau'$, and $\gamma=\gamma'$. So the $h$-minimizing measure $\mu$ is supported on the closed geodesic $\gamma$. Thus there exists $\lambda \in \R$ such that 
$h = \lambda \left[\gamma\right]$. Since $\gamma$ is a simple closed curve in $M\setminus \Gamma$, there exists a closed curve $\alpha$ in $M\setminus \Gamma$ which intersects $\gamma$ exactly once, so $\inter ( h, \left[\alpha\right]) = \lambda$. We also have 
$\inter ( h_i, \left[\alpha\right]) = 0$ because $\gamma$ lies in $M\setminus \Gamma$. Since $h = l_1 + a_1 h_1, + \ldots + a_{g-1}h_{g-1}$, it follows that $\inter ( l_1, \left[\alpha\right])= \lambda$. Now, since $l_1 \in H_1 (M,\Z)$, this entails $\lambda \in \Z$. 
On the other hand, since $\mu + \mu_{\Gamma}$ is minimizing, we have $\gamma \in \mathcal{G}(\Gamma)$, so 
$\lambda \gamma \in \mathcal{G}(\Gamma)$. Since $\lambda \left[\gamma\right] \mod V(\Gamma) =l$, this proves the surjectivity of $\psi_{\Gamma}$.
\qed

For any $t \geq 0$, denote by $N_{\Gamma}(t)$ the number of elements of $L(\Gamma)$ which are contained in the canonical projection of $\mathcal{F}_t(\Gamma)$ to $V(\Gamma)^{\perp}/ V(\Gamma) $.

\begin{lemma}\label{cardinaux G Gamma et L Gamma}
Let $M$ be a closed, orientable, surface of genus $g$, equipped with a generic Finsler metric $m$ and let $\gamma_1, \ldots \gamma_{g-1}$ be simple closed geodesics such that the formal sum $\Gamma := \gamma_1 + \ldots + \gamma_{g-1}$ is a minimizing multicurve.

Then, for any $t \geq 0$, $N_{\Gamma}(t)$ equals the number $\sharp \mathcal{G}_t(\Gamma)$ of elements of $\mathcal{G}(\Gamma)$ of length less than $t$.
\end{lemma}
\proof
We prove the lemma by showing that the map $\alpha \mapsto \left[\alpha\right] \mod V(\Gamma)$ is a one-to-one correspondance between 
$$
L(\Gamma) \bigcap \left( \mathcal{F}_t(\Gamma) \mod V(\Gamma) \right)
$$
and the subset $\mathcal{G}_t(\Gamma)$ of $\mathcal{G}(\Gamma)$ which consists of closed geodesics of length $\leq t$.

Take $\alpha \in \mathcal{G}_t(\Gamma)$. Then $ \left[\alpha\right] \in \mathcal{F}(\Gamma) $, and $\|\left[\alpha\right]\| \leq \mbox{length}(\alpha) \leq t$, so $ \left[\alpha\right] \in \mathcal{F}_t(\Gamma) $. Hence 
$$
 \left[\alpha\right] \mod V(\Gamma) \in \left( \mathcal{F}_t(\Gamma) \mod V(\Gamma) \right).
 $$
Besides, $ \left[\alpha\right] \in H_1(M,\Z)$ so $ \left[\alpha\right] \mod V(\Gamma) \in L(\Gamma)$. Thus 
$$
 \left[\alpha\right] \mod V(\Gamma) \in \left( L(\Gamma) \bigcap \left( \mathcal{F}_t(\Gamma) \mod V(\Gamma) \right) \right).
 $$
 Conversely, take $h$ in $L(\Gamma) \bigcap \left( \mathcal{F}_t(\Gamma) \mod V(\Gamma) \right)$.
Since $h \in L(\Gamma)$, by Lemma \ref{injection G Gamma} there exists a unique $\alpha$ in $ \mathcal{G}(\Gamma) $ such that 
$ \left[\alpha\right] =h\mod V(\Gamma)$. Now $h \in \mathcal{F}_t(\Gamma) \mod V(\Gamma)$, so there exists $h' \in \mathcal{F}_t(\Gamma)$ such that 
$h'= h\mod V(\Gamma)$. But Proposition \ref{injection F Gamma} says there exists a unique $h'' \in \mathcal{F}(\Gamma)$ such that 
$h''= h\mod V(\Gamma)$; since $ \left[\alpha\right] \in \mathcal{F}(\Gamma)$ and $ \left[\alpha\right] =h\mod V(\Gamma)$, it follows that 
$ \left[\alpha\right] =h'=h''$, so $ \left[\alpha\right] \in \mathcal{F}_t(\Gamma) $. Therefore $\|\left[\alpha\right]\| \leq t$. But the closed geodesic $\alpha$ is minimizing since it lies in $\mathcal{G}(\Gamma)$, so $\|\left[\alpha\right]\| = \mbox{length}(\alpha) $, thus $\mbox{length}(\alpha) \leq t$, which finishes the proof of the lemma.
\qed

%

\begin{theorem}
Let 
\begin{itemize}[itemsep=2ex,leftmargin=0.5cm]
 \item $M$ be a closed, orientable, surface of genus $g$, equipped with a generic Finsler metric $m$
 \item $\gamma_1, \ldots \gamma_{g-1}$ be simple closed geodesics such that the formal sum $\Gamma := \gamma_1 + \ldots + \gamma_{g-1}$ is a minimizing multicurve
 \item $ \mathcal{G}_t (\Gamma)$ be the set of closed geodesics $\alpha$ of length $\leq t$, such that $\alpha +\Gamma$ is a minimizing multicurve.
\end{itemize}
Then,  $\mbox{Leb}$ being  the normalized Lebesgue measure in $\R^n$, 
$$
\lim_{t \rightarrow + \infty} \frac{1}{t^2}\sharp \mathcal{G}_t (\Gamma) = \mbox{Leb} \mathcal{P}_1 (\Gamma)
$$
\end{theorem}
\proof
By Lemma \ref{cardinaux G Gamma et L Gamma}, for any $t \geq 0$, 
$\sharp \mathcal{G}_t (\Gamma) $ equals the number of elements of $L(\Gamma)$ which are contained in the projection $\mathcal{P}_t(\Gamma)$ of
 $\mathcal{F}_t(\Gamma)$ to $V^{\perp}(\Gamma) / V(\Gamma)$. Observe that for any $t\geq 0$, we have 
 $$
 \mathcal{P}_t(\Gamma)= t \mathcal{P}_1(\Gamma).
 $$
 By Lemma \ref{bord negligeable}, $\mathcal{P}_1(\Gamma)$ is a compact subset of $V^{\perp}(\Gamma) / V(\Gamma)$, whose boundary is Lebesgue-negligible. Then Proposition \ref{appendice} says that the number of elements of $L(\Gamma)$ which are contained in $\mathcal{P}_t(\Gamma)$, divided by $t^2$, converges to the Lebesgue measure of $\mathcal{P}_1 (\Gamma)$. 
\qed


\appendix

\section{}
The classical Minkovsky theorem says that if $C$ is a convex body in $\R^n$, then 
$$
\lim_{t \rightarrow + \infty} \frac{1}{t^n} \sharp \left( tC \cap \Z^n \right) = \mbox{ Leb } (C)
$$
where $\sharp$ denotes the cardinality of a set and $\mbox{Leb}$ is the normalized Lebesgue measure in $\R^n$. Here we need to deal with non-convex bodies so we prove the following
\begin{proposition}\label{appendice}
Let $C$ be a compact subset of $\R^2$ such that the boundary $\partial C$ of $C$ has zero Lebesgue measure. Then 
$$
\lim_{t \rightarrow + \infty} \frac{1}{t^2} \sharp \left( tC \cap \Z^2 \right) = \mbox{ Leb } (C).
$$
\end{proposition}
\proof
For $t >0$ we define a measure $\mu_t$ on $\R^2$ by 
$$
\mu_t := \frac{1}{t^n} \sum_{z \in \Z^2} \delta (\frac{z}{t})
$$
where $\delta(z) $ is the Dirac measure at $z$. Observe that 
$$
\mu_t (C) = \frac{1}{t^2} \sharp \left( tC \cap \Z^2 \right)
$$
so we want to prove that $\lim_{t \rightarrow + \infty} \mu_t (C) = \mbox{ Leb } (C)$.
Let us first evaluate $\mu_t$ on rectangles. Take real numbers $a \leq b, c \leq d$. We have 
\begin{eqnarray}
\mu_t \left( \left[ a,b\right] \times \left[ c,d\right] \right) &\leq & \frac{1}{t^2} \left( E(t(b-a)) +1 \right) \left( E(t(d-c)) +1 \right) \\
 &\leq & \frac{1}{t^2} \left( t(b-a) +1 \right) \left( t(d-c) +1 \right) \\
 &\leq & \left( b-a + \frac{1}{t} \right) \left( d-c +\frac{1}{t} \right) \label{appendice, maj}
 \end{eqnarray}
 and similarly
\begin{eqnarray}
\mu_t \left( \left] a,b\right[ \times \left] c,d\right[ \right) &\geq & \left( E(t(b-a)) -1 \right) \left( E(t(d-c)) -1 \right)\\
 &\geq & \frac{1}{t^2} \left( t(b-a) -2 \right) \left( t(d-c) -2 \right) \\
 &\geq & \left( b-a - \frac{2}{t} \right) \left( d-c -\frac{2}{t} \right). \label{appendice, min}
 \end{eqnarray}
 Now pick $\epsilon >0$. Since $\mbox{ Leb } (\partial C)=0$, we may cover the compact set $\partial C$ by a a finite number of open rectangles $R_1, \ldots R_k$ such that 
 \begin{equation}\label{appendice, leq epsilon}
 \mbox{ Leb } (R_1 \cup \ldots \cup R_k) < \epsilon.
 \end{equation}
Take open rectangles $R_{k+1}, \ldots R_p$ contained in $C$ such that $C \subset R_1 \cup \ldots \cup R_p$. Then 
\begin{eqnarray*}
\mu_t \left( R_{k+1}\cup \ldots \cup R_p \right) &\leq & \mu_t (C) \leq \mu_t \left( R_{1}\cup \ldots \cup R_p \right)\\
 \mbox{ Leb }\left( R_{k+1}\cup \ldots \cup R_p \right) &\leq & \mbox{ Leb } (C) \leq \mbox{ Leb } \left( R_{1}\cup \ldots \cup R_p \right)
\end{eqnarray*} 
which, by the inequality (\ref{appendice, leq epsilon}), entails 
\begin{equation}\label{appendice, leb compris}
\mbox{ Leb } \left( R_{1}\cup \ldots \cup R_p \right) -\epsilon \leq \mbox{ Leb } (C) \leq \mbox{ Leb }\left( R_{k+1}\cup \ldots \cup R_p \right) + \epsilon
\end{equation}

Now by the inequalities (\ref{appendice, maj}, \ref{appendice, min}), for each $i=1, \ldots p$, $\mu_t (R_i) $ converges to 
$ \mbox{ Leb } (R_i)$, so there exists a real number $T$ such that for all $t \geq T$, we have
\begin{eqnarray*}
\mu_t \left( R_{k+1}\cup \ldots \cup R_p \right) & \geq & \mbox{ Leb } \left( R_{k+1}\cup \ldots \cup R_p \right) -\epsilon \\
\mu_t \left( R_{1}\cup \ldots \cup R_p \right) & \leq & \mbox{ Leb } \left( R_{1}\cup \ldots \cup R_p \right) +\epsilon .
\end{eqnarray*}
Then by (\ref{appendice, leb compris}), for all $t \geq T$, we have
$$
| \mu_t (C) - \mbox{ Leb }(C) | \leq 4 \epsilon
$$
which, $\epsilon$ being arbitrary, proves the convergence of $\mu_t (C)$ to $ \mbox{ Leb }(C)$, and the proposition.
\qed 
 
 {\small

\bigskip

\noindent

D\'epartement de Math\'ematiques, Universit\'e Montpellier 2, France\\
e-mail: massart@math.univ-montp2.fr

\noindent

D\'epartement de Math\'ematiques, Universit\'e de Fribourg, Suisse\\
e-mail: hugo.parlier@unifr.ch


\begin{thebibliography}{99}
%
\bibitem[BM08]{nonor} F. Balacheff, D. Massart {\em Stable norms of non-orientable surfaces}, 
Ann. Inst. Fourier (Grenoble), 58 no. 4 (2008), p. 1337-1369
%
\bibitem[BC08]{BC} P. Bernard, G. Contreras 
{\em A generic property of families of Lagrangian systems } Annals of Math., 167, no. 3 (2008)
%
\bibitem[F74]{Federer} H. Federer, H {\em Real flat chains, cochains and
variational problems} Ind. Univ. Math. J. {\bf 24}, 351--407 (1974)
%
\bibitem[GLP]{GLP} M. Gromov, J. Lafontaine, P.Pansu {\em Structures m{\'e}triques pour les vari{\'e}t{\'e}s riemanniennes} CEDIC-Fernand Nathan. Paris, (1981)
%
\bibitem[KH]{Katok} A. Katok, B. Hasselblatt {\em Introduction to the modern theory of dynamical systems} Encyclopedia of Mathematics and its Applications, 54. Cambridge University Press, Cambridge, 1995.
%
\bibitem[KMS86]{KMS} S. Kerckhoff, H. Masur, J. Smillie
{\em Ergodicity of billiard flows and quadratic differentials}
Ann. of Math. (2) 124 (1986), no. 2, 293-311. 
%
\bibitem[Mn92]{Mane92} 
R. Ma\~n\'e 
{\em On the minimizing measures of Lagrangian dynamical systems} 
Nonlinearity {\bf 5} (1992), no. 3, 623--638.
%
\bibitem[Mn96]{Mane} 
R. Ma\~n\'e 
{\em Generic properties and problems of
minimizing measures of Lagrangian systems} 
Nonlinearity {\bf 9} (1996), no. 2, 273--310.
%
\bibitem[M69]{Margulis}
G. Margulis {\em On some applications of ergodic theory to the study of manifolds on negative curvature} Fundam. Anal. Appl. 3 (1969), 89-90.
%
\bibitem[Mt97]{gafa} D. Massart 
{\em Stable norms of surfaces: local structure of the unit ball at rational directions} Geom. Funct. Anal. 7 (1997), {\bf 6}, 996-1010.
%
\bibitem[Ms90]{Masur} H. Masur {\em The growth rate of trajectories of a quadratic differential} Ergodic Theory Dynam. Systems 10 (1990), no. 1, 151-176.
%
\bibitem[Mr91]{Mather91} J. N. Mather 
{\em Action minimizing invariant measures for positive definite 
Lagrangian systems} Math. Z. {\bf 207}, 169-207 (1991).
%
\bibitem[McSP08]{McShane-Parlier} G. McShane, H. Parlier
{\em Multiplicities of simple closed geodesics and hypersurfaces in Teichm\H{u}ller space}, Geom. Topol. , vol. 12 (2008) no. 4, pp. 1883-1919

%
\bibitem[McSR95i]{McS-R1} G. McShane, I. Rivin
{\em A norm on homology of surfaces and counting simple geodesics}, 
Internat. Math. Res. Notices, {\bf 2} (1995), 61--69.
%
\bibitem[McSR95ii]{McS-R2} G. McShane, I. Rivin
{\em Simple curves on hyperbolic tori} C. R. Acad. Sci. Paris S\'er. I Math. 320 (1995), no. 12, 1523-1528. 

%
\bibitem[PS87]{Philips-Sarnak} R. Phillips, P. Sarnak {\em Geodesics in homology classes} Duke Math. J. 55 (1987), 287-297.
%
\end{thebibliography}
\end{document}